\documentclass[11pt,reqno]{amsart}
\usepackage{color}
\usepackage{amsmath,amssymb,amsthm,amsfonts}
\usepackage{hyperref}
\usepackage{color}
\usepackage{setspace}
\usepackage{placeins}
\makeatletter

\newcommand{\Rmnum}[1]{\expandafter\@slowromancap\romannumeral #1@}
\makeatother

\newtheorem{thm}{Theorem}[section]

\newtheorem{lemma}[thm]{Lemma}
\newtheorem{remark}{Remark}[section]
\newtheorem{theorem}[thm]{Theorem}

\usepackage[numbers,sort&compress]{natbib}
\usepackage{graphicx}
\usepackage{caption}
\usepackage{subcaption}
\usepackage{tikz}
\usepackage{multirow}

\usepackage{lineno}

\begin{document}

\author{Huaizhi Cao}
\address{School of Mathematics, South China University of Technology,
Guangzhou 510640, China}
\email{mahzcao@163.com}

\author{Jiawei Chu$^*$}
\address{School of Mathematics, South China University of Technology,
Guangzhou 510640, China}
\email{majwchu@163.com}

\author{Haiyang Jin}
\address{School of Mathematics, South China University of Technology,
Guangzhou 510640, China}
\email{mahyjin@scut.edu.cn}

\title[A parabolic-elliptic forager-exploiter model]{Parabolic-elliptic reduction suppresses Hopf bifurcation in a forager-exploiter system with cascade taxis}

\begin{abstract}
We study a parabolic-elliptic forager-exploiter model describing the cascade interaction between two species through a shared environmental resource with a constant renewal rate. We first establish the global existence and uniform boundedness of classical solutions in arbitrary dimensions {\color{black}for large initial population data and large taxis sensitivity coefficients. We then investigate the long-time behavior of solutions.} When the resource renewal rate is small, all solutions will converge exponentially to the unique constant steady state. Furthermore, for large resource renewal rates, the  parabolic-elliptic system does not undergo Hopf bifurcation, which is in sharp contrast to the corresponding fully parabolic case, where temporal oscillations arise via Hopf bifurcation.
\end{abstract}


\thanks{$^*$Corresponding author: Jiawei Chu}

\maketitle

\numberwithin{equation}{section}

\section{Introduction and Main Results}
\subsection{Introduction}
{\color{black}The classical prey-taxis model describes direct predator–prey interactions and fails to generate spatial patterns \cite{Lee-JBD-2009,JinWang-JDE-2017,Wu-JDE-2016}.
Moreover, it cannot capture indirect resource-searching behavior, where one species acquires information about resource through another species.} Field observations of Alaskan seabirds reveal a typical two-tiered ecological interplay \cite{7}: kittiwakes (foragers) actively aggregate at high-resource zones, while shearwaters (exploiters) indirectly locate food by following kittiwake flocks. To mathematically describe this cascaded foraging interaction between two species sharing a common resource, Tania et al. \cite{9} introduced the forager–exploiter reaction-diffusion model as follows:
\begin{equation}\label{system*}
\begin{cases}
u_t=\Delta u-\chi_1\nabla\cdot(u\nabla w),&x\in\Omega,~t>0,\\
v_t=\Delta v-\chi_2\nabla\cdot(v\nabla u),&x\in\Omega,~t>0,\\
w_t=d\Delta w-\lambda(u+v)w-\mu w+r(x,t),&x\in\Omega,~t>0,\\
\nabla u\cdot\nu=\nabla v\cdot\nu=\nabla w\cdot \nu=0,&x\in\partial\Omega,~t>0,\\
(u,v, w)(x,0)=(u_0, v_0, w_0)(x), &x\in\Omega,
\end{cases}
\end{equation}
where $\Omega\subset\mathbb{R}^n(n\geq 1)$ is a bounded domain with smooth boundary,  the functions $u=u(x,t),v=v(x,t),w=w(x,t)$ denote the densities of the forager population, the exploiter population and the resource at position $x$ and time $t$, respectively.  The taxis terms $-\chi_1\nabla\cdot(u\nabla w)$ and $-\chi_2\nabla\cdot (v\nabla u)$ represent a cascade movement mechanism: foragers move toward areas with higher resource density, whereas exploiters move toward areas with higher forager density. The positive parameters $\chi_1,\chi_2$ are taxis sensitivity coefficients. The external resource is replenished through the nonnegative function $r(x,t)\geq 0$. The homogeneous Neumann boundary conditions describe no-flux across the boundary and all other parameters are positive constants.

The cascade taxis structure makes the mathematical analysis more challenging, especially in higher spatial dimensions.
Most existing analytical results for \eqref{system*} concern global existence, boundedness and the global stability of constant steady states, but are restricted either to one dimension or to higher-dimensional settings ($n\geq2$) under suitable smallness assumptions. Specifically, in one dimension, the global boundedness and global stabilization of constant steady state for the system \eqref{system*} with constant resource renewal rates were first established by Tao and Winkler \cite{11}. In higher dimensions, the existence of generalized global solutions to \eqref{system*} was obtained under an explicit relation between $w_0$ and $r(x,t)$, while global stability of constant steady state requires additional temporal decay conditions on $r(x,t)$ \cite{16}. Later, classical solvability was established only under smallness assumptions on initial data $u_0, w_0$, the resource term $r(x,t)$, or the taxis sensitivity coefficients $\chi_1,\chi_2$ \cite{13}. Under such smallness or decay assumptions, all solutions converge to the constant steady state and no spatiotemporal patterns occur. Taken together, the majority of existing analytical contributions to \eqref{system*} focus on global well-posedness and the asymptotic stability of constant equilibria, while rigorous theories for pattern formation and periodic dynamics are far less complete. We refer to \cite{2,3,4,13,14,15,17} and the references therein for more comprehensive summaries of analytical results concerning the full parabolic system \eqref{system*}.
In contrast, numerical simulations in \cite{9} demonstrated that a sufficiently large constant resource renewal rate (i.e., $r(x,t)\equiv r>0$ with large constant $r$) induces oscillatory spatiotemporal patterns. Very recently, Cao et al. \cite{CaoChuJin-2026} further provided rigorous theoretical verification for the global existence and asymptotic stability of time-periodic solutions to the full parabolic system \eqref{system*} in arbitrary dimensions under both homogeneous/heterogeneous resource environments, and clarified how distinct renewal mechanisms shape periodic dynamics. However, these stability results all rely on the prior existence of global classical solutions. To date, rigorous results on global well-posedness of \eqref{system*} for large initial population data or large taxis sensitivity coefficients are still lacking, which substantially limits the applicability of the existing theory.

To overcome such theoretical limitations while remaining biologically relevant, we consider the following parabolic‑elliptic variant of \eqref{system*}:
\begin{equation}\label{system}
\begin{cases}
u_t=\Delta u-\chi_1\nabla\cdot(u\nabla w),&x\in\Omega,~t>0,\\
v_t=\Delta v-\chi_2\nabla\cdot(v\nabla u),&x\in\Omega,~t>0,\\
0=d\Delta w-\lambda(u+v)w-\mu w+r(x,t),&x\in\Omega,~t>0,\\
\nabla u\cdot\nu=\nabla v\cdot\nu=\nabla w\cdot \nu=0,&x\in\partial\Omega,~t>0,\\
(u,v)(x,0)=(u_0, v_0)(x), &x\in\Omega.
\end{cases}
\end{equation}
In fact, in many biological and ecological situations, nutrients or chemical cues diffuse much faster than the movement of individuals.  This gives rise to the well‑known quasi‑stationary approximation.
Such a parabolic-elliptic reduction has been extensively used in chemotaxis-type systems because the elliptic equation enjoys stronger regularizing effects and often facilitates the analysis of global classical solutions (see, e.g.,  \cite{Bellomo-Winkler-2017-TAMS, Perthame-2007, Luckhaus-1992-TAMS, Nagai-1995, TaoWinkler-JDE-2019-PE}).  However, this reduction also introduces a nonlocal coupling through the elliptic operator and therefore may change the structure of the linearized operator from that of the corresponding fully parabolic system. As a consequence, the dynamics of the resulting parabolic–elliptic system may substantially differ from those of the fully parabolic system.  Motivated by these considerations, we address the following two questions:
\begin{itemize}
\item whether the parabolic-elliptic structure guarantees global classical solvability for the two-species forager–exploiter systems in higher dimensions
{\color{black}for large initial population data and taxis sensitivity coefficients;}
\item how the parabolic-elliptic structure affects the {\color{black}dynamics} compared with the fully parabolic system.
\end{itemize}
Since the heterogeneous cases can be treated by arguments essentially identical to those in \cite{CaoChuJin-2026}, we restrict ourselves to homogeneous environments throughout this paper, i.e., $r(x,t)\equiv r$ with $r$ being a positive constant. Moreover, we assume that $\Omega\subset\mathbb{R}^n$ ($n\geq 1$) is a bounded domain with a smooth boundary, and impose the following basic hypothesis on initial data:
\begin{equation}\label{IC}
   u_0\in W^{1,\infty}(\Omega), v_0\in C^0(\overline{\Omega}) ~\text{with}~ u_0\geq,\not\equiv0,v_0\geq,\not\equiv0.
\end{equation}
\subsection{Main results}
In homogeneous environments, the fully parabolic system \eqref{system*} and the parabolic-elliptic system \eqref{system} exhibit dramatically different dynamical behaviors. Before presenting our results, we note that {\color{black}for any fixed $u_0$ and $v_0$,} \eqref{system} admits a unique positive constant steady state $(\bar{u}_0,\bar{v}_0,w_c)$ with
$w_c:=\frac{r}{\lambda(\bar{u}_0+\bar{v}_0)+\mu},$
where the notation $\bar{f}:=\frac{1}{|\Omega|}\int_\Omega f dx$.  

\vspace{1.5mm}
Our first theorem establishes the global classical solvability and the global stability of the constant steady state $(\bar{u}_0,\bar{v}_0,w_c)$.
\vspace{2mm}

\begin{theorem}\label{global_tau0_app}
Let $\Omega \subset \mathbb{R}^n(n \ge 1)$ be a bounded domain with a smooth boundary, and \eqref{IC} hold.  Then there exists a constant $r_0>0$ such that if $0<r\leq r_0$, the system \eqref{system} admits a unique global nonnegative classical solution $(u,v,w)\in  [C^0(\overline{\Omega} \times [0,\infty)) \cap C^{2,1}(\overline{\Omega} \times (0,\infty))]^2\times  C^{2,0}(\overline{\Omega} \times (0,\infty))$ with a uniform-in-time bound in the sense that 
\begin{equation*}
\|u(\cdot, t)\|_{W^{1,\infty}} + \|v(\cdot, t)\|_{L^\infty} + \|w(\cdot, t)\|_{W^{1,\infty}} \le C_0 \quad \forall t > 0.
\end{equation*}
Moreover, we have 
$$\|u(\cdot,t)-\bar{u}_0\|_{L^\infty}+\|v(\cdot,t)-\bar{v}_0\|_{L^\infty}+\|w(\cdot,t)-w_c\|_{L^\infty}\leq C_1 e^{-\kappa_0 t}, \ \forall t\geq1,$$
where the constants $C_0, C_1$ and $\kappa_0$ are independent of $t$.
\end{theorem}

\begin{remark}
\em{Theorem \ref{global_tau0_app} establishes the global existence and uniform boundedness of classical solutions for \eqref{system} without the smallness assumptions on the initial data or $\chi_1,\chi_2$ imposed in \cite{13}. In fact, when $r=r(x,t) \geq,\not\equiv0$ is non-constant, and $r(x,t)\in C^{\alpha,\frac{\alpha}{2}}(\overline{\Omega}\times[0,\infty))$  with $\alpha\in (0,1)$ is globally bounded, one can obtain the same global existence and boundedness results for \eqref{system}.} 
\end{remark}

Theorem \ref{global_tau0_app} rules out any other nontrivial long-time behaviors when $r$ is small. A natural question is therefore whether more complicated dynamics may emerge when the resource renewal rate becomes large.  To answer this question, we first linearize the system \eqref{system} at $(\bar{u}_0,\bar{v}_0,w_c)$ to get the following linearized problem:
\begin{equation}\label{le*0}
\begin{cases}
(\psi_1)_t=\Delta \psi_1-\chi_1\bar{u}_0\Delta \psi_3,&x\in\Omega,t>0,\\
(\psi_2)_t=\Delta\psi_2-\chi_2\bar{v}_0\Delta\psi_1,&x\in\Omega,t>0,\\
0=d\Delta \psi_3-\lambda w_c(\psi_1+\psi_2)-J\psi_3,&x\in\Omega,t>0,\\
\nabla\psi_1\cdot\nu=\nabla\psi_2\cdot\nu=\nabla\psi_3\cdot\nu=0, &x\in\partial\Omega,t>0,\\
(\psi_1,\psi_2)(x,0)=(u_0-\bar{u}_0,v_0-\bar{v}_0), &x\in\Omega,\\
\int_\Omega \psi_1(\cdot,t)=\int_\Omega\psi_2(\cdot,t)=0,
& t>0,\\
\end{cases}
\end{equation}
where $\psi_1:=u-\bar{u}_0,~\psi_2:=v-\bar{v}_0,~\psi_3:=w-w_c,$ and $J:=\lambda(\bar{u}_0+\bar{v}_0)+\mu>0.$ Since the third equation is elliptic, $\psi_3$ is uniquely determined by $(\psi_1,\psi_2)$ and therefore does not constitute an independent dynamical degree of freedom. Eliminating $\psi_3$ yields the reduced system 
\begin{equation}\label{le*0-1}
\begin{cases}
(\psi_1)_t=\Delta \psi_1+\chi_1\bar{u}_0\lambda w_c\Delta (-d\Delta+J)^{-1}(\psi_1+\psi_2),&x\in\Omega,t>0,\\
(\psi_2)_t=\Delta\psi_2-\chi_2\bar{v}_0\Delta\psi_1,&x\in\Omega,t>0,\\
\nabla\psi_1\cdot\nu=\nabla \psi_2\cdot\nu=0, &x\in\partial\Omega,t>0,\\
(\psi_1,\psi_2)(x,0)=(u_0-\bar{u}_0,v_0-\bar{v}_0), &x\in\Omega,\\
\int_\Omega \psi_1(\cdot,t)=\int_\Omega\psi_2(\cdot,t)=0.
& t>0.
\end{cases}
\end{equation}
Hence, the linearized dynamics of \eqref{le*0} are completely described by that of \eqref{le*0-1}. However, the reduced linearized system \eqref{le*0-1} is governed by a nonlocal operator involving $\Delta(-d\Delta +J)^{-1}$ rather than the standard Neumann Laplacian.  Therefore, its stability cannot be inferred directly from an eigenvalue analysis. Instead, one must first establish that the corresponding elliptic operator generates an analytic semigroup and has compact resolvent, so that the spectral bound determines the growth of the linearized flow and the spectrum can be characterized entirely through eigenvalues.

\vspace{2mm}
The above discussion reveals that establishing linear stability for \eqref{le*0-1} requires a nonstandard spectral analysis {\color{black} for the associated nonlocal operator. By developing the corresponding spectral framework together with a direct argument for steady states, we obtain the following result.
}
\begin{theorem}\label{NPS}
Let $\chi_1,\chi_2,\lambda,\mu,d,\bar{v}_0,\bar{u}_0$ be fixed. Then for all $r>0$,  the constant steady state $(\bar{u}_0,\bar{v}_0,w_c)$ is linearly stable, and \eqref{system} does not admit positive non-constant steady state solutions.
\end{theorem}

\begin{remark}
\em{ {\color{black}The linear stability in Theorem \ref{NPS} excludes Hopf bifurcation emanating from $(\bar{u}_0,\bar{v}_0,w_c)$.} It reveals a striking contrast between the parabolic-elliptic system \eqref{system} and the fully parabolic system \eqref{system*}, for which temporal oscillations arise through Hopf bifurcation (c.f. \cite{CaoChuJin-2026}). }

\end{remark}

\vspace{2mm}
The remaining part of this paper is organized as follows. In Section \ref{app}, we establish the global existence, boundedness and stability of classical solutions in arbitrary dimensional settings. Section \ref{LB} analyzes the spectrum of the nonlocal linearized operator, establishes linear stability of  $(\bar u_0,\bar v_0,w_c)$, and proves the nonexistence of positive non-constant steady states.

\section{Global boundedness and stabilization: proof of Theorem \ref{global_tau0_app}}\label{app}
In what follows, we shall abbreviate $\int_\Omega f dx$ and $\|f\|_{L^p(\Omega)}$ as $\int_\Omega f$ and $\|f\|_{L^p}$, respectively.  The symbols $c_i$, $C_i(i=1,2,3\cdots)$ are used to denote generic positive constants which are independent of $t$ and may vary in the context. 
\subsection{Local existence and preliminaries}
This subsection will establish the local existence of solutions to \eqref{system}. 
\begin{lemma}[Local existence]\label{locale}
Let $\Omega \subset \mathbb{R}^n(n \ge 1)$ be a bounded domain with a smooth boundary, and \eqref{IC} hold. Then there exists $T_{\operatorname{max}} \in (0, \infty]$ such that \eqref{system} admits a unique positive classical solution $(u,v,w)$ fulfilling
\begin{equation*}
u, v \in C^0(\overline{\Omega} \times [0,T_{\operatorname{max}})) \cap C^{2,1}(\overline{\Omega} \times (0,T_{\operatorname{max}})), \quad w \in C^{2,0}(\overline{\Omega} \times (0,T_{\operatorname{max}})).
\end{equation*}
Moreover, if $T_{\operatorname{max}} < \infty$, then 
\begin{equation}\label{bu}
\limsup_{t \nearrow T_{\operatorname{max}}} \left( \|u(\cdot, t)\|_{W^{1,p}} + \|v(\cdot, t)\|_{L^\infty} \right) = \infty,\ \ \text{for~some~} p>n.
\end{equation}
\end{lemma}
\begin{proof}
We shall use the Banach fixed point theorem to prove the local existence. The proof will be divided into three steps.

{\bf Step 1: Defining a map $\Phi$.} 
Fix $p>n$. For any given $\widetilde{T} \in (0,1)$, we define the Banach space 
$$X_{\widetilde{T}} := C^0([0,\widetilde{T}]; W^{1,p}(\Omega)) \times C^0([0,\widetilde{T}]; C^0(\overline{\Omega}))$$
equipped with the norm $$\|(u,v)\|_{X_{\widetilde{T}}} := \|u\|_{L^\infty([0,\widetilde{T}]; W^{1,p})} + \|v\|_{L^\infty([0,\widetilde{T}]; L^\infty(\Omega))}.$$ 
By standard smoothing estimates for the Neumann heat semigroup \(\{e^{t\Delta}\}_{t\ge 0}\), there exists a constant \(c_1>0\) such that
$$
\|e^{t\Delta}u_0\|_{W^{1,p}} \le c_1\|u_0\|_{W^{1,p}}.
$$
We further introduce a closed subset \(S\subset X_{\widetilde{T}}\)  as follows:
\begin{equation*}
    S := \left\{ (u,v) \in X_{\widetilde{T}} \;\big|\; \|u\|_{L^\infty((0,\widetilde{T}); W^{1,p}(\Omega))} \le N_1, \; \|v\|_{L^\infty((0,\widetilde{T}); L^\infty(\Omega))} \le N_2, \; \min_{\overline{\Omega} \times [0,\widetilde{T}]} (u + v) \ge -\frac{\mu}{2\lambda} \right\},
\end{equation*}
where the positive constants
$$N_1 := c_1\|u_0\|_{W^{1,p}} + 1 \ \ \mathrm{and} \ \ N_2 := \|v_0\|_{L^\infty} + 1.$$
For any given $(u,v)\in S$, since $\min_{\overline{\Omega}\times[0,\widetilde T]}(u+v)
\ge -\frac{\mu}{2\lambda},$ we have
\[
\lambda(u+v)+\mu
\ge \lambda\left(-\frac{\mu}{2\lambda}\right)+\mu
=\frac{\mu}{2}>0.
\]
Therefore, by the Lax--Milgram theorem, there exists
a unique weak solution $w\in H^1(\Omega)$ solving
\begin{equation}\label{we}
\begin{cases}
d\Delta w-(\lambda(u+v)+\mu)w=-r,
&x\in\Omega ,\\
\nabla w\cdot\nu=0,
&x\in\partial\Omega.
\end{cases}
\end{equation}
Moreover, the maximum principle yields $w\ge0$. Applying standard $L^p$-elliptic regularity theory to \eqref{we}, then using the Sobolev embedding $W^{2,p}(\Omega)\hookrightarrow W^{1,\infty}(\Omega)$ for $p>n$, we obtain 
\begin{equation}\label{eqwb}
\|w(\cdot,t)\|_{L^\infty}
+\|\nabla w(\cdot,t)\|_{L^\infty}
\le c_2\|w(\cdot,t)\|_{W^{2,p}}
\le c_3\|-r\|_{L^\infty}
\le c_4,
\end{equation}
where $c_4>0$ is a constant depending on $N_1$ and $N_2$. 
Moreover, the continuous dependence of the elliptic solution on
$(u,v)$ implies that
\begin{equation}\label{w-cont}
w\in C^0([0,\widetilde T];W^{2,p}(\Omega)).    
\end{equation}
Then we can define a mapping $\Phi(u,v)(t) := (\Phi_1(u,v)(t), \Phi_2(u,v)(t))$ via Duhamel's formula:
\begin{align*}
    \Phi_1(u,v)(t) &:= e^{t\Delta}u_0 - \chi_1 \int_0^t e^{(t-s)\Delta} F(u(s), v(s), w(s)) ds, \\
    \Phi_2(u,v)(t) &:= e^{t\Delta}v_0 - \chi_2 \int_0^t e^{(t-s)\Delta} \nabla \cdot (v(s) \nabla u(s)) ds,
\end{align*}
where 
\begin{equation*}
    F(u,v,w) := \nabla u\cdot \nabla w+u\Delta w =\nabla u \cdot \nabla w + \frac{u}{d} \big( \lambda(u+v)w + \mu w - r \big).
\end{equation*}

{\bf Step 2: Showing $\Phi(S) \subset S$.} By the definition of the set $S$ and the Sobolev embedding theorem, one has
$$\|u\|_{L^\infty} \le c_5 \|u\|_{W^{1,p}} \le c_5 N_1,$$
which, combined with \eqref{eqwb}, gives
\begin{equation}\label{EF}
    \|F(u,v,w)\|_{L^{p}} \le c_4N_1  + \frac{c_5 N_1|\Omega|^{\frac{1}{p}}}{d} \big( \lambda(c_5 N_1 + N_2)c_4 + \mu c_4 + r \big)=:c_6.
\end{equation}
 Then using the well-known semigroup estimates (see e.g., \cite[Lemma 2.1]{CaoXinru-DCDSA-2015}) and \eqref{EF}, we derive 
\begin{equation}\label{phi1}
\begin{split}
    \|\Phi_1(u,v)(t)\|_{W^{1,p}} &\le c_1\|u_0\|_{W^{1,p}} + c_7\chi_1 \int_0^t  \big(1 + (t-s)^{-\frac{1}{2}}\big) \|F(s)\|_{L^{p}} ds \\
    &\le c_1\|u_0\|_{W^{1,p}} + c_6 c_7 \chi_1 (\widetilde{T} + 2\widetilde{T}^{\frac{1}{2}}) \\
    &\le c_1\|u_0\|_{W^{1,p}} + c_8 \widetilde{T}^{\frac{1}{2}},
\end{split}
\end{equation}
where $c_8 > 0$ depends on $N_1$ and $N_2$. Similarly, using the semigroup estimates again, and noting the fact that $\|v \nabla u\|_{L^{p}} \le \|v\|_{L^\infty} \|\nabla u\|_{L^{p}} \le N_2 N_1,$ we have 
\begin{equation}\label{phi2}
\begin{split}
    \|\Phi_2(u,v)(t)\|_{L^\infty} 
    &\le \|e^{t\Delta} v_0\|_{L^\infty} + c_9\chi_2 \int_0^t \big(1 + (t-s)^{-\frac{1}{2} - \frac{n}{2p}}\big) e^{-\lambda_1 (t-s)} \|v(s) \nabla u(s)\|_{L^{p}} ds \\
    &\le \|v_0\|_{L^\infty} + c_9 \chi_2 N_1 N_2 \left( \widetilde{T} + \frac{\widetilde{T}^{1-\gamma}}{1-\gamma} \right) \\
    &\le \|v_0\|_{L^\infty} + c_{10} \widetilde{T}^{1-\gamma},
\end{split}
\end{equation}
where $\gamma := \frac{1}{2} + \frac{n}{2p} = \frac{p+n}{2p} < 1$ (since $p>n$),  and $c_{10} > 0$ depends on  $N_1$ and $N_2$. 

Choosing 
$
T_1:=\min\{c_8^{-2}, c_{10}^{-\frac{1}{1-\gamma}},1\},
$
then for any $\widetilde{T}\in (0,T_1)$, we deduce from \eqref{phi1} and \eqref{phi2} that
\begin{equation}\label{phi12}
\|\Phi_1(u,v)(t)\|_{L^\infty ((0,\widetilde{T});W^{1,p}(\Omega))}\leq c_1\|u_0\|_{W^{1,p}}+1=N_1, 
\end{equation}
and
\begin{equation}\label{phi12-1}
\|\Phi_2(u,v)(t)\|_{L^\infty((0,\widetilde{T});L^\infty(\Omega))}\leq \|v_0\|_{L^\infty}+1=N_2.
\end{equation}
Noting that $e^{t\Delta}(u_0 + v_0) \ge 0$ since $u_0, v_0 \ge 0$, and using the Sobolev embedding $W^{1,p}(\Omega) \hookrightarrow L^{\infty}(\Omega)$ (since $p > n$) again, we have
\begin{equation}\label{Lb}
\begin{split}
    \Phi_1(u,v)(t) + \Phi_2(u,v)(t) &\ge 0 - \|\Phi_1(t) - e^{t\Delta}u_0\|_{L^\infty} - \|\Phi_2(t) - e^{t\Delta}v_0\|_{L^\infty} \\
    &\ge - c_5\|\Phi_1(t) - e^{t\Delta}u_0 \|_{W^{1,p}} -\|\Phi_2(t) - e^{t\Delta}v_0\|_{L^\infty} \\
    &\geq - c_5c_8 \widetilde{T}^{\frac{1}{2}} - c_{10} \widetilde{T}^{1-\gamma}.
\end{split}
\end{equation}
 Define 
\begin{equation*}
    T_2 := \min \left\{T_1,\  \frac{1}{2}, \ \big(\frac{\mu}{4\lambda c_5 c_8}\big)^2, \ \big(\frac{\mu}{4\lambda c_{10}}\big)^{\frac{1}{1-\gamma}} \right\},
\end{equation*}
then for any $\widetilde{T}\in (0,T_2)$, some calculations imply
\begin{equation*}
    c_5 c_8 \widetilde{T}^{\frac{1}{2}} + c_{10} \widetilde{T}^{1-\gamma} \le  c_5 c_8 T_2^{\frac{1}{2}} + c_{10} T_2^{1-\gamma} \le \frac{\mu}{4\lambda} + \frac{\mu}{4\lambda} = \frac{\mu}{2\lambda},
\end{equation*}
which, substituted into \eqref{Lb}, gives
\begin{equation}\label{Lb-1}
\displaystyle\min_{\overline{\Omega} \times [0,\widetilde{T}]} (\Phi_1(u,v)(t) + \Phi_2(u,v)(t)) \ge -\frac{\mu}{2\lambda}.    
\end{equation}
Combining \eqref{w-cont} with the smoothing and strong continuity properties of the Neumann heat semigroup, and using \eqref{phi12}, \eqref{phi12-1}, and \eqref{Lb-1}, we conclude that $\Phi(S)\subset S.$

{\bf Step 3: Establishing the contraction property.} For any $(u,v)$, $(\hat{u}, \hat{v})\in S$, we denote the corresponding solutions to \eqref{we} by $w$ and $\hat{w}$, respectively, which satisfy \eqref{eqwb}. Then $z := w - \hat{w}$ satisfies
\begin{equation*}
\begin{cases}
    d\Delta z - (\lambda(u+v) + \mu)z
    = \lambda(u - \hat{u} + v - \hat{v})\hat{w},
    & x\in\Omega,~t>0,\\
    \nabla z\cdot\nu=0,
    & x\in\partial\Omega,~t>0.
\end{cases}
\end{equation*}
By the embedding $W^{2,p}(\Omega) \hookrightarrow W^{1,\infty}(\Omega) \ (p>n)$, the standard $L^p$-elliptic estimates, the H\"{o}lder inequality and \eqref{eqwb}, one  derives
\begin{equation}\label{eq:diff_z}
\begin{split}
    \|z\|_{W^{2,p}} + \|z\|_{L^\infty} + \|\nabla z\|_{L^\infty} 
    &\le (1+c_2)\|z\|_{W^{2,p}}\\
    &\leq c_{11}\|\lambda(u - \hat{u} + v - \hat{v})\hat{w}\|_{L^p}\\
    &\leq \lambda c_4 c_{11} \big( \|u - \hat{u}\|_{L^{p}} + |\Omega|^{\frac{1}{p}} \|v - \hat{v}\|_{L^\infty} \big)\\
    &\leq c_{12} \|(u,v) - (\hat{u}, \hat{v})\|_{X_{\widetilde{T}}}.
\end{split}
\end{equation}
Noting \eqref{eqwb} and using the facts $\|\hat{u}\|_{L^\infty}\leq c_5\|\hat{u}\|_{W^{1,p}}\leq c_5N_1$ and $\|\hat{v}\|_{L^\infty}\leq N_2$, then from \eqref{eq:diff_z}, we can find a constant $c_{13}>0$ depending on $N_1$ and $N_2$ such that
\begin{align*}
&\|F(u,v,w) - F(\hat{u}, \hat{v}, \hat{w})\|_{L^p}\\ 
&= \|\nabla u \cdot \nabla z + \nabla (u - \hat{u}) \cdot \nabla \hat{w} +u\Delta z + (u-\hat{u})\frac{1}{d} \big[ \lambda(\hat{u}+\hat{v})\hat{w}+\mu \hat{w}-r \big]\|_{L^p}\\
&\leq \|\nabla z\|_{L^\infty}\|\nabla u\|_{L^p}+\|\nabla (u-\hat{u})\|_{L^p}\|\nabla \hat{w}\|_{L^\infty}+\|\Delta z\|_{L^p}\|u\|_{L^\infty}\\
&\ \ \ \ \ +\frac{\|u-\hat{u}\|_{L^p}}{d}\|\lambda(\hat{u}+\hat{v})\hat{w}+\mu \hat{w}-r\|_{L^\infty}\\
&\leq c_{13}  \|(u,v) - (\hat{u}, \hat{v})\|_{X_{\widetilde{T}}}.
\end{align*}
Applying the same semigroup estimates and using the fact $\widetilde{T} < 1$, we have
\begin{equation}\label{contrac-1}
\begin{aligned}
    \|\Phi_1(u,v)(t) - \Phi_1(\hat{u}, \hat{v})(t)\|_{W^{1,p}(\Omega)} 
    &\le \chi_1 \int_0^t \big\| e^{(t-s)\Delta} \big( F(u,v,w) - F(\hat{u}, \hat{v}, \hat{w}) \big) \big\|_{W^{1,p}(\Omega)} ds \\
    &\le c_7 c_{13} \chi_1 \|(u,v) - (\hat{u}, \hat{v})\|_{X_{\widetilde{T}}}\int_0^t \big(1 + (t-s)^{-\frac{1}{2}}\big) ds \\
    &\le c_{14} \widetilde{T}^{\frac{1}{2}} \|(u,v) - (\hat{u}, \hat{v})\|_{X_{\widetilde{T}}}.
\end{aligned}
\end{equation}
Similarly, we have
\begin{align*}
    \|\Phi_2(u,v)(t) - \Phi_2(\hat{u}, \hat{v})(t)\|_{L^\infty} 
    &\le \chi_2 \int_0^t \big\| e^{(t-s)\Delta} \nabla \cdot \big( v(s)\nabla u(s) - \hat{v}(s)\nabla \hat{u}(s) \big) \big\|_{L^\infty} ds \\
    &\le c_9 \chi_2 \int_0^t \big(1 + (t-s)^{-\gamma}\big) \big( N_2 \|u - \hat{u}\|_{W^{1,p}} + N_1 \|v - \hat{v}\|_{L^\infty} \big) ds \\
    &\le c_{15} \widetilde{T}^{1-\gamma} \|(u,v) - (\hat{u}, \hat{v})\|_{X_{\widetilde{T}}},
\end{align*}
which, combined with \eqref{contrac-1}, implies
\begin{equation*}
    \|\Phi(u,v) - \Phi(\hat{u}, \hat{v})\|_{X_{\widetilde{T}}} \le c_{16} \max\big\{ \widetilde{T}^{\frac{1}{2}}, \widetilde{T}^{1-\gamma} \big\} \|(u,v) - (\hat{u}, \hat{v})\|_{X_{\widetilde{T}}}.
\end{equation*}
Set 
\begin{equation*}
    T_3 := \min \big\{ (2c_{16})^{-2}, \ (2c_{16})^{-\frac{1}{1-\gamma}}, T_2\big\}.
\end{equation*}
Then for any $\widetilde{T} \in (0, T_3]$, it holds that
\begin{equation*}
    \|\Phi(u,v) - \Phi(\hat{u}, \hat{v})\|_{X_{\widetilde{T}}} \le \frac{1}{2} \|(u,v) - (\hat{u}, \hat{v})\|_{X_{\widetilde{T}}}.
\end{equation*}
Thus, the combination of Steps 1--3 shows that if $\widetilde{T}\in (0,T_3),$ the mapping $\Phi (S)\subset S$ is a strict contraction on $S$.  By the Banach fixed-point theorem,  we find a $(u,v)\in S$ such that 
$$\Phi(u,v)=(u,v),$$
and $w$ is uniquely determined by \eqref{we}. Furthermore, the standard parabolic and elliptic Schauder regularity theories ensure that $(u,v,w)$ is a classical solution in $\Omega\times (0, \widetilde{T})$, and the positivity of the solution follows from maximum principle. This along with the arbitrariness of $u_0\in W^{1,p}(\Omega)$ and $v_0\in C^0({\overline{\Omega}})$ means that for any $N_1,N_2$, there exists $\widetilde{T}:=\widetilde{T}(N_1,N_2)$ small enough such that if $\|u_0\|_{W^{1,p}}\leq N_1, \|v_0\|_{L^\infty}\leq N_2,$
then \eqref{system} is classically solvable in $\Omega\times (0, \widetilde{T})$.

A standard extension argument yields the existence of $T_{\operatorname{max}}$ fulfilling \eqref{bu}. The proof of uniqueness for the classical solutions of the system \eqref{system} is standard, we omit the details. 
\end{proof}
\subsection{Global boundedness} In this subsection, we shall prove the global boundedness of solution for the system \eqref{system}. Before that, we first give some basic  estimates.
\begin{lemma} 
Let $(u,v,w)$ be the solution obtained in Lemma \ref{locale}. Then it holds that
\begin{equation}\label{pb}
 \int_\Omega u(\cdot,t)= \int_\Omega u_0, \ \ \  \int_\Omega v(\cdot,t)= \int_\Omega v_0,\quad\forall t\in(0,T_{\operatorname{max}}),
\end{equation}
and 
\begin{equation}\label{lw2}
0\leq w(x,t)\leq \frac{r}{\mu},
\quad \forall (x,t)\in
\bar{\Omega}\times(0,T_{\operatorname{max}}).
\end{equation}
\end{lemma}
\begin{proof}
Direct integrations and the comparison principle give \eqref{pb} and \eqref{lw2}, respectively.
\end{proof}
\begin{lemma}\label{appui}
Let $(u,v,w)$ be the solution obtained in Lemma \ref{locale}. Then there exists a constant $C_1>0$ independent of $r$ such that 
\begin{equation}\label{appu_infty}
    \sup_{t \in (0, T_{\max})} \|u(\cdot, t)\|_{L^\infty} \le C_1 (r^{\frac{n+2}{2}} + 1).
\end{equation}
\end{lemma}

\begin{proof}
 Multiplying the first equation of \eqref{system} by $q u^{q-1}$ for $q \ge 2$ and integrating  it over $\Omega$, we obtain
 \begin{equation}\label{u-i1}
\begin{split}
    \frac{d}{d t} \int_\Omega u^q + \frac{4(q-1)}{q} \int_\Omega \big|\nabla u^{\frac{q}{2}}\big|^2 
    &= \chi_1 q(q-1) \int_\Omega u^{q-1} \nabla u \cdot \nabla w \\
    &= -\chi_1 (q-1) \int_\Omega u^q \Delta w \\
    &= -\frac{\chi_1 (q-1)}{d} \int_\Omega u^q \big( \lambda(u+v)w + \mu w - r \big) \\
    &\le \frac{\chi_1 r(q-1)}{d} \int_\Omega u^q , 
\end{split}
\end{equation}
where we have used the  identity $u^{q-2}|\nabla u|^2=\frac{4}{q^2}|\nabla u^{\frac{q}{2}}|^2$  and the fact $\Delta w = \frac{1}{d}(\lambda(u+v)w + \mu w - r)$.

Let $c_1=\frac{\chi_1}{d}$. Then, by the Gagliardo-Nirenberg inequality and Young's inequality, one derives
\begin{equation}\label{0722*}
\begin{split}
\left(\frac{\chi_1(q-1)}{d} r +1\right)\int_\Omega u^q  
&\leq   (c_1 r q +1)\int_\Omega u^q \\
&\leq (c_1 r + 1) q \|u^{\frac{q}{2}}\|_{L^2}^2 \\
&\le c_2 (c_1 r + 1) q \|\nabla u^{\frac{q}{2}}\|_{L^2}^{\frac{2n}{n+2}} \|u^{\frac{q}{2}}\|_{L^1}^{\frac{4}{n+2}} + c_3 (c_1 r + 1) q \|u^{\frac{q}{2}}\|_{L^1}^2\\
&\le \|\nabla u^{\frac{q}{2}}\|_{L^2}^2 + c_4 (c_1 r + 1)^{\frac{n+2}{2}} q^{\frac{n+2}{2}} \|u^{\frac{q}{2}}\|_{L^1}^2.
\end{split}
\end{equation}
Substituting \eqref{0722*} into \eqref{u-i1} and using \(\frac{4(q-1)}{q} > 1\) for \(q \ge 2\), we obtain
\begin{equation}\label{0723}
    \frac{d}{dt} \int_\Omega u^q + \int_\Omega u^q \le c_4 (c_1 r + 1)^{\frac{n+2}{2}} q^{\frac{n+2}{2}} \left( \int_\Omega u^{\frac{q}{2}} \right)^2.
\end{equation}
Let $q = 2^j$ for $j \ge 1$, then applying the ODE comparison principle to \eqref{0723} yields
\begin{equation}\label{moser-0}
    \sup_{t \in (0, T_{\max})} \int_\Omega u^{2^j} \le \int_\Omega u_0^{2^j} + c_4 (c_1 r + 1)^{\frac{n+2}{2}} 2^{\frac{n+2}{2}j} \sup_{t \in (0, T_{\max})} \left( \int_\Omega u^{2^{j-1}} \right)^2.
\end{equation}
{\color{black}Define 
\begin{equation}\label{defineMj}
    M_j := \max\big\{ 1, \|u_0\|_{L^\infty}, \sup_{t \in (0, T_{\max})} \|u(\cdot, t)\|_{L^{2^j}} \big\}.
\end{equation}}
Noting that $\int_\Omega u_0^{2^j}\le |\Omega|\|u_0\|_{L^\infty}^{2^j}$, we infer from \eqref{pb}, \eqref{moser-0} and \eqref{defineMj} that there exists a constant $c_5>0$, independent of $j$ and $r$, such that
\begin{equation}\label{0722*1}
\begin{split}
M_j &\le c_5^{\frac{1}{2^j}} (r + 1)^{\frac{n+2}{2^{j+1}}} 2^{\frac{n+2}{2} \frac{j}{2^j}} M_{j-1} \\
&\le c_5^{\sum_{i=1}^j \frac{1}{2^i}} (r + 1)^{\frac{n+2}{2} \sum_{i=1}^j \frac{1}{2^i}} 2^{\frac{n+2}{2} \sum_{i=1}^j \frac{i}{2^i}} M_0 \\
&= c_5^{1 - \frac{1}{2^j}} (r + 1)^{\frac{n+2}{2} \left( 1 - \frac{1}{2^j} \right)} 2^{\frac{n+2}{2} \left( 2 - \frac{j+2}{2^j} \right)} M_0.
\end{split}
\end{equation}
By \eqref{pb}, one gets that $M_0 = \max\left\{ 1, \|u_0\|_{L^\infty},  \|u_0\|_{L^1} \right\}$ is bounded by a positive constant independent of $r$. Then we deduce from \eqref{0722*1} that
\begin{equation*}
    \sup_{t \in (0, T_{\max})} \|u(\cdot, t)\|_{L^\infty} \leq \lim_{j \to \infty} M_j \le c_5 (r + 1)^{\frac{n+2}{2}} 2^{n+2} M_0 \le C_1(r^{\frac{n+2}{2}} + 1). 
\end{equation*}
Hence the proof of Lemma \ref{appui} is complete.
\end{proof}

\begin{lemma}\label{appvp}
Let $(u,v,w)$ be the solution obtained in Lemma \ref{locale}. There exists a constant $r_1\in(0,1)$ such that for any given $p>n$ and $0<r\leq r_1$, there exist two constants $C_2, C_3>0$, independent of $t$ and $r$, such that
\begin{equation}\label{appv_p}
    \sup_{t \in (0,T_{\max})} \|v(\cdot, t)\|_{L^p} \le C_2,
\end{equation}
and 
\begin{equation}\label{appuv}
\|w(\cdot,t)\|_{W^{1,\infty}}+\|\nabla u(\cdot,t)\|_{L^\infty} \leq C_3, \quad  \forall t \in (0,T_{\operatorname{max}}).  
\end{equation}
\end{lemma}
\begin{proof}
Let $M>0$ be a constant satisfying $\|v_0\|_{L^p}<M$ (its value will be specified later), and define
\begin{equation}\label{Tdef}
    \widehat{T}:= \sup \big\{ \widetilde{T} \in (0, T_{\max}) \,\big|\, \|v(\cdot, t)\|_{L^p} < M \text{ for all } t \in (0, \widetilde{T}) \big\}.
\end{equation}
Since $\|v(\cdot, 0)\|_{L^p} < M$ and the map $t \mapsto \|v(\cdot, t)\|_{L^p}$ is continuous, the quantity $\widehat{T}$ is positive and well defined. We shall show $\widehat{T}=T_{\operatorname{max}}$ under the smallness assumption on $r$. 

Without loss of generality, assume that $0<r \le  1$.  Throughout the proof, the constants $c_i$ appearing below are independent of $M$ and $r$.  By \eqref{appu_infty}, there exists a constant $K_1>0$ independent of $M$ and $r$ such that 
\begin{equation}\label{appu_infty2}
\sup_{t \in (0, T_{\max})} \|u(\cdot, t)\|_{L^\infty} \le K_1(r^{\frac{n+2}{2}} + 1) \le 2K_1.  
\end{equation}
Multiplying the second equation of \eqref{system} by $pv^{p-1}$, then integrating the results by parts, and applying  Young's inequality as well as the identity $\big|\nabla v^{\frac{p}{2}}\big|^2 = \frac{p^2}{4}v^{p-2}|\nabla v|^2$, we derive
\begin{equation}\label{vo}
\begin{aligned}
\frac{d}{dt} \int_\Omega v^p + \frac{4(p-1)}{p} \int_\Omega \big| \nabla v^{\frac{p}{2}} \big|^2
&=\chi_2 p(p-1) \int_\Omega v^{p-1} \nabla v \cdot \nabla u \\   
&\leq \frac{2(p-1)}{p} \int_\Omega \big| \nabla v^{\frac{p}{2}} \big|^2+\frac{\chi_2^2p(p-1)}{2}\|\nabla u\|_{L^\infty}^2\int_\Omega v^p.
\end{aligned}
\end{equation}
By the Gagliardo-Nirenberg inequality and the fact $\int_\Omega v  \equiv \int_\Omega v_0$ (see \eqref{pb}), we have 
\begin{equation*}
\begin{aligned}
\int_\Omega v^p = \big\|v^{\frac{p}{2}}\big\|_{L^2}^2 
&\leq c_1 \big\|\nabla v^{\frac{p}{2}}\big\|_{L^2}^{\frac{2n(p-1)}{np + 2 - n}} \big\|v^{\frac{p}{2}}\big\|_{L^{\frac{2}{p}}}^{2-\frac{2n(p-1)}{np + 2 - n}} + c_{2} \big\|v^{\frac{p}{2}}\big\|_{L^{\frac{2}{p}}}^2 \leq c_{3} \left( \int_\Omega \big| \nabla v^{\frac{p}{2}} \big|^2 \right)^{\frac{n(p-1)}{np + 2 - n}} + c_{4},
\end{aligned}
\end{equation*}  
which updates \eqref{vo} as
\begin{equation}\label{vo1}
\begin{split}
\frac{d}{dt} \int_\Omega v^p + \int_\Omega v^p + \frac{2(p-1)}{p} \int_\Omega \big| \nabla v^{\frac{p}{2}} \big|^2
 &\leq A(t) \int_\Omega v^p\\
 &\leq c_3 A(t) \left( \int_\Omega \big| \nabla v^{\frac{p}{2}} \big|^2 \right)^{\frac{n(p-1)}{np + 2 - n}}+c_4A(t),
\end{split}
\end{equation}
where 
\begin{equation}\label{A}
A(t):=\frac{\chi_2^2p(p-1)}{2}\|\nabla u(\cdot,t)\|_{L^\infty}^2+1.
\end{equation}
Applying Young's inequality to \eqref{vo1} gives
\begin{equation}\label{appv1}
\begin{split}
 \frac{d}{dt} \int_\Omega v^p + \int_\Omega v^p 
&\leq c_{5} A^{\frac{np + 2 - n}{2}}(t) + c_{4}A(t).
\end{split}
\end{equation}
Then we solve \eqref{appv1} to derive 
\begin{equation}\label{vpm-1}
 \|v\|_{L^p}^p \leq \|v_0\|_{L^p}^p+e^{-t}\int_0^te^{s}\big(c_{5} A^{\frac{np + 2 - n}{2}}(s) + c_{4}A(s)\big)ds.
\end{equation}
To estimate $A(t)$, we need to consider $\|\nabla u(\cdot,t)\|_{L^\infty}$. First, the semigroup estimate implies 
$$\|\nabla e^{t\Delta} u_0\|_{L^\infty}\leq c_0\|\nabla u_0\|_{L^\infty},$$ then we take
$$M:=\|v_0\|_{L^p}+c_{5}^{\frac{1}{p}}\hat{A}^{\frac{np+2-n}{2p}}+c_{4}^{\frac{1}{p}}\hat{A}^{\frac{1}{p}}+2,$$
where 
\begin{equation}\label{hatA}
\hat{A}:=2\chi_2^2p(p-1)c_0^2(\|\nabla u_0\|_{L^\infty}+1)^2+1.
\end{equation}
By the standard $L^p$-elliptic regularity theory and the Sobolev embedding $W^{2,p}(\Omega) \hookrightarrow W^{1,\infty}(\Omega)$, and using \eqref{lw2},\eqref{Tdef} and \eqref{appu_infty2} , we deduce from the third equation of \eqref{system} that, for all $t\in(0,\widehat{T})$,
\begin{equation}\label{ndw}
\begin{split}
\|w\|_{W^{1,\infty}} \leq c_6\| w \|_{W^{2,p}} \le c_7 \|r-\lambda(u+v)w\|_{L^p}
&\leq c_8\left(r|\Omega|^{\frac{1}{p}}+\frac{\lambda(M+2K_1|\Omega|^{\frac{1}{p}})r}{\mu}\right)\\
&\leq c_9(M+1)r.
\end{split}
\end{equation}
 Applying Duhamel's formula to the first equation of \eqref{system} gives
\begin{equation*}
    u(\cdot, t) = e^{t\Delta} u_0 - \chi_1 \int_0^t e^{(t-s)\Delta} \nabla \cdot (u(\cdot, s) \nabla w(\cdot, s))ds,
\end{equation*}
hence 
\begin{equation*}
\begin{aligned}
\|\nabla u(\cdot, t)\|_{L^\infty} 
&\le \|\nabla e^{t\Delta} u_0\|_{L^\infty}  + \chi_1 \int_0^t \|\nabla e^{(t-s)\Delta} \nabla \cdot (u(\cdot, s) \nabla w(\cdot, s))\|_{L^\infty}  ds. \\
\end{aligned} 
\end{equation*}
Then we use the semigroup estimates, \eqref{appu_infty2} and \eqref{ndw} to derive that  for all $t\in(0,\hat T)$, 
\begin{equation}\label{appnu1}
\begin{aligned}
&\|\nabla u(\cdot, t)\|_{L^\infty}\\ 
&\le c_0 \|\nabla u_0\|_{L^\infty} +c_{10}\chi_1\int_0^t(1+(t-s)^{-\frac{1}{2}-\frac{n}{2p}})e^{-\lambda_1(t-s)}\|\nabla \cdot (u(\cdot, s) \nabla w(\cdot, s))\|_{L^p}ds \\
&\le c_{10}\chi_1c_9(M+1)r\bigg(|\Omega|^{\frac1p}\sup\limits_{t\in (0,\widehat{T})}\|\nabla u(\cdot,t)\|_{L^\infty}+\frac{2K_1}{c_6}\bigg)\int_0^t(1+(t-s)^{-\frac{1}{2}-\frac{n}{2p}})e^{-\lambda_1(t-s)}ds\\
&\quad + c_0 \|\nabla u_0\|_{L^\infty}\\
&\leq c_0 \|\nabla u_0\|_{L^\infty}+ c_9c_{10}c_{11}\chi_1(M+1)r\big(\sup\limits_{t\in (0,\widehat{T})}\|\nabla u(\cdot,t)\|_{L^\infty}+{2K_1}/{c_6}\big),
\end{aligned} 
\end{equation}
where we have used the fact 
$$
(1+|\Omega|^{\frac{1}{p}})\int_0^t(1+(t-s)^{-\frac{1}{2}-\frac{n}{2p}})e^{-\lambda_1(t-s)}ds\leq c_{11}.
$$
Taking $r\leq r_1=:\min\{\frac{1}{2 c_9c_{10}c_{11}\chi_1(M+1)},\frac{c_0c_6}{2K_1c_{9}c_{10}c_{11}\chi_1(M+1)},1\}$, it follows from \eqref{appnu1} that
\begin{equation}\label{appnu2}
\begin{split}
\sup_{t \in (0, \hat T)}\|\nabla u(\cdot,t)\|_{L^\infty}
&\leq 2c_0 \|\nabla u_0\|_{L^\infty}+\frac{4 K_1c_9c_{10}c_{11}\chi_1}{c_6}(M+1)r\\
&\leq 2c_0 (\|\nabla u_0\|_{L^\infty}+1).
\end{split}
\end{equation}
Hence, by the definitions of $A$ and $\hat{A}$ in \eqref{A} and \eqref{hatA}, respectively, we deduce from \eqref{vpm-1} that
\begin{equation}\label{07231}
\|v\|_{L^p}\leq \|v_0\|_{L^p}+c_{5}^{\frac{1}{p}} \hat{A}^{\frac{np + 2 - n}{2p}} + c_{4}^{\frac{1}{p}}\hat{A}^{\frac{1}{p}}< M-1,\ \forall t\in (0,\widehat{T}).
\end{equation}
If $\widehat{T} < T_{\max}$, the continuity of the norm implies $\|v(\cdot, \widehat{T})\|_{L^p} = M$, which contradicts \eqref{07231}. Therefore, we obtain $\widehat{T} = T_{\max}$ and 
\begin{equation*}
\|v(\cdot, t)\|_{L^p} \leq M - 1, \forall t \in (0, T_{\max}),
\end{equation*}
which proves \eqref{appv_p}. Finally, \eqref{appuv} directly follows from \eqref{ndw} and \eqref{appnu2}.
\end{proof}
\begin{lemma}[Global boundedness]\label{GB}
{\color{black}Let the conditions in Lemma \ref{locale} hold.} Then there exists a constant $r_1>0$ such that if $0<r\leq r_1$, system \eqref{system} admits a unique global classical solution $(u,v,w)\in  [C^0(\overline{\Omega} \times [0,\infty)) \cap C^{2,1}(\overline{\Omega} \times (0,\infty))]^2\times  C^{2,0}(\overline{\Omega} \times (0,\infty))$. Moreover, there exists a constant $K_2>0$ independent of $t$ and $r$ such that 
\begin{equation}\label{GB-1}
\|u(\cdot, t)\|_{W^{1,\infty}} + \|v(\cdot, t)\|_{L^\infty} + \|w(\cdot, t)\|_{W^{1,\infty}} \le K_2 \quad \text{for all } t > 0.
\end{equation}
\end{lemma}
\begin{proof}
{\color{black}By Lemma \ref{appvp}, we know there exists a constant $r_1>0$ such that if $0<r\leq r_1$,  \eqref{appv_p} and \eqref{appuv} hold for all  $t\in(0, T_{\max})$.} We now estimate the $L^\infty$-norm of $v$. Applying Duhamel’s formula to the second equation of \eqref{system} gives
\begin{equation*}
 v(\cdot, t) =  e^{t\Delta} v_0 - \chi_2 \int_0^t  e^{(t-s)\Delta} \nabla \cdot (v(\cdot, s) \nabla u(\cdot, s)) ds,    
\end{equation*}
which, together with the semigroup estimate and \eqref{appuv}, implies
\begin{equation}\label{appvinfty}
\begin{aligned}
\|v(\cdot, t)\|_{L^\infty} 
&\le \|e^{t\Delta} v_0\|_{L^\infty} + \chi_2 \int_0^t \big\| e^{(t-s)\Delta} \nabla \cdot (v(\cdot, s) \nabla u(\cdot, s)) \big\|_{L^\infty} ds \\
&\le \|v_0\|_{L^\infty} + c_2 \chi_2 \int_0^t \big(1 + (t-s)^{-\frac{1}{2} - \frac{n}{2p}}\big) e^{-\lambda_1(t-s)} \|v(\cdot, s)\|_{L^p} \|\nabla u(\cdot, s)\|_{L^\infty} ds \\
&\le \|v_0\|_{L^\infty} + c_2 \chi_2 C_2C_3 \int_0^\infty \big(1 + s^{-\frac{1}{2} - \frac{n}{2p}}\big) e^{-\lambda_1 s} ds \leq c_3.
\end{aligned}
\end{equation}
Hence the combination of \eqref{bu},  \eqref{appuv}, \eqref{appu_infty2} with \eqref{appvinfty} finishes the proof of Lemma \ref{GB}.
\end{proof}

\begin{lemma}\label{lem:holder}
{\color{black}Let $(u, v, w)$ be the nonnegative global classical solution of \eqref{system} obtained in Lemma \ref{GB}.} Then there exists a constant $C_{4} > 0$ independent of $t$ and $r$ such that
\begin{equation}\label{conclu_uv}
\|v\|_{C^{\alpha, \frac{\alpha}{2}}(\bar{\Omega}\times[t, t+1])} \le C_{4},  \quad \forall t \ge 1
\end{equation}
for some constant $\alpha\in(0,1).$
\end{lemma}
\begin{proof}
We rewrite the second equation of \eqref{system} as
$$ v_t = \nabla \cdot A_1(x, t, v, \nabla v),$$
where $A_1(x, t, v, \nabla v) = \nabla v - \chi_2 v \nabla u.$ By \eqref{GB-1}, we have
\begin{equation}\label{A1_cond1}
\begin{aligned}
A_1(x, t, v, \nabla v) \cdot \nabla v &= (\nabla v - \chi_2 v \nabla u) \cdot \nabla v 
\ge \frac{1}{2}|\nabla v|^2 - \frac{\chi_2^2}{2} v^2 |\nabla u|^2 
\ge \frac{1}{2}|\nabla v|^2 - \frac{K_2^4\chi_2^2}{2},
\end{aligned}
\end{equation}
and
\begin{equation}\label{A1_cond2}
|A_1(x, t, v, \nabla v)| \le |\nabla v| + \chi_2 |v| |\nabla u| \le |\nabla v| + K_2^2\chi_2.
\end{equation}
With \eqref{A1_cond1} and \eqref{A1_cond2}, we apply the H\"{o}lder regularity for quasilinear parabolic equations \cite[Theorem 1.3 and Remark 1.4]{PV-JDE-1993} to obtain $\|v\|_{C^{\alpha, \frac{\alpha}{2}}(\bar{\Omega}\times[t, t+1])} \le c_3$ for all $t\geq 1$. Hence we complete the proof of Lemma \ref{lem:holder}.
\end{proof}
\begin{lemma}[Global stabilization]\label{GST=0}
Let $(u,v,w)$ be the solution obtained in Lemma \ref{GB}. Then there exists a small constant $r_2>0$ such that  when $0<r\leq r_2$, it holds that
\begin{equation}\label{GST=0-1}
\|u(\cdot,t)-\bar{u}_0\|_{L^\infty}+\|v(\cdot,t)-\bar{v}_0\|_{L^\infty}+\|w(\cdot,t)-w_c\|_{L^\infty}\leq C_5e^{-\kappa_1 t}, \ \forall t\geq1,
\end{equation}
for some constants $C_5$ and $\kappa_1$ independent of $t$ and $r.$
\end{lemma}
\begin{proof}
We multiply the first equation in \eqref{system} by $(u-\bar{u}_0)$, and integrate the results by parts, then use Young's inequality and \eqref{GB-1} to derive
\begin{equation*}
\begin{split}
\frac{1}{2}\frac{d}{dt}\int_\Omega (u-\bar{u}_0)^2 +\int_\Omega |\nabla (u-\bar{u}_0)|^2
&=\chi_1\int_\Omega u\nabla w\cdot \nabla (u-\bar{u}_0)\\
&\leq \frac{1}{2}\int_\Omega |\nabla (u-\bar{u}_0)|^2 +\frac{\chi_1^2K_2^2}{2}\int_\Omega |\nabla w|^2,
\end{split}
\end{equation*}
which implies 
\begin{equation}\label{GST=0-2}
\frac{d}{dt}\int_\Omega (u-\bar{u}_0)^2 +\int_\Omega |\nabla (u-\bar{u}_0)|^2\leq \chi_1^2K_2^2\int_\Omega |\nabla w|^2.
\end{equation}
Similarly, we deduce from the $v$-equation in \eqref{system} that
\begin{equation}\label{GST=0-3}
\frac{d}{dt}\int_\Omega (v-\bar{v}_0)^2 +\int_\Omega |\nabla (v-\bar{v}_0)|^2
\leq \chi_2^2K_2^2\int_\Omega |\nabla u|^2.
\end{equation}
Multiplying the third equation in \eqref{system} by $w-w_c$ and  integrating it by parts, we have
\begin{equation*}
\begin{split}
d\int_\Omega |\nabla (w-w_c)|^2+\int_\Omega [\mu+\lambda(u+v)](w-w_c)^2
&= -\lambda w_c\int_\Omega [(u-\bar{u}_0)+(v-\bar{v}_0)](w-w_c)\\
&\leq \frac{\mu}{2}\int_\Omega (w-w_c)^2+\frac{\lambda^2 w_c^2}{\mu}\int_\Omega [(u-\bar{u}_0)^2+(v-\bar{v}_0)^2],
\end{split}
\end{equation*}
{\color{black}which, along with the fact that $r=\lambda (\bar{u}_0+\bar{v}_0)w_c+\mu w_c$, gives} 
\begin{equation}\label{GST=0-4}
d\int_\Omega |\nabla (w-w_c)|^2+\frac{\mu}{2}\int_\Omega (w-w_c)^2\leq c_0r^2\int_\Omega [(u-\bar{u}_0)^2+(v-\bar{v}_0)^2],
\end{equation}
where $c_0:=\frac{\lambda^2 }{\mu[\lambda(\bar{u}_0+\bar{v}_0)+\mu]^2}$.

On the other hand, since $\int_\Omega (u-\bar{u}_0)=\int_\Omega(v-\bar{v}_0)=0$, we apply the Poincar\'e inequality to find a constant $c_p>0$ satisfying
\begin{equation}\label{poincare-A}
\int_\Omega |u-\bar{u}_0|^2 \le c_p \int_\Omega |\nabla (u-\bar{u}_0)|^2 \quad \text{and} \quad \int_\Omega |v-\bar{v}_0|^2 \le c_p \int_\Omega |\nabla (v-\bar{v}_0)|^2.
\end{equation}
Multiplying \eqref{GST=0-2} and \eqref{GST=0-4} by $2\chi_2^2 K_2^2$ and  $\frac{4\chi_1^2 \chi_2^2 K_2^4}{d}$, respectively, and adding the results to \eqref{GST=0-3}, then using \eqref{poincare-A}, we obtain
\begin{equation}\label{GST=0-5}
\begin{split}
&\frac{d}{dt} \left( 2\chi_2^2 K_2^2 \int_\Omega (u-\bar{u}_0)^2 + \int_\Omega (v-\bar{v}_0)^2 \right) \\
&\leq -c_1(r) \int_\Omega (u-\bar{u}_0)^2- c_2(r) \int_\Omega (v-\bar{v}_0)^2  - \frac{2\mu \chi_1^2 \chi_2^2 K_2^4}{d} \int_\Omega (w-w_c)^2,
\end{split}
\end{equation}
where $c_1(r)$ and $c_2(r)$ are defined as
\begin{equation*}
c_1(r) := \frac{\chi_2^2 K_2^2}{c_p} - \frac{4c_0 r^2 \chi_1^2 \chi_2^2 K_2^4}{d}, \quad 
c_2(r) := \frac{1}{c_p} - \frac{4c_0 r^2 \chi_1^2 \chi_2^2 K_2^4}{d}.
\end{equation*}
Then there exists a small constant $r_2\in(0,r_1)$ such that for all $0<r\leq r_2$, it holds that
\begin{equation*}
c_1(r) \geq \frac{\chi_2^2 K_2^2}{2c_p} \quad \text{and} \quad c_2(r) \geq \frac{1}{2c_p},
\end{equation*}
which, substituted into \eqref{GST=0-5}, gives
\begin{equation}\label{GST=0-6}
\begin{split}
\frac{d}{dt} \left( 2\chi_2^2 K_2^2 \int_\Omega (u-\bar{u}_0)^2 + \int_\Omega (v-\bar{v}_0)^2 \right)+ \frac{\chi_2^2 K_2^2}{2c_p} \int_\Omega (u-\bar{u}_0)^2 + \frac{1}{2c_p} \int_\Omega (v-\bar{v}_0)^2  \leq 0.
\end{split}
\end{equation}
Let $E(t) := 2\chi_2^2 K_2^2 \int_\Omega (u-\bar{u}_0)^2 + \int_\Omega (v-\bar{v}_0)^2$ and  $\kappa_1 := \frac{1}{4c_p}$, we derive from \eqref{GST=0-6} that
\begin{equation*}
\frac{d}{dt}E(t) + \kappa_1 E(t) \leq 0.
\end{equation*}
Applying Gronwall's inequality, we immediately obtain
\begin{equation}\label{GST=L2-uv}
\|u(\cdot,t)-\bar{u}_0\|_{L^2} + \|v(\cdot,t)-\bar{v}_0\|_{L^2} \leq c_3 e^{-\frac{\kappa_1}{2}t},
\end{equation}
 which, together with \eqref{GST=0-4}, implies
\begin{equation}\label{GST=L2-w}
\|w(\cdot,t)-w_c\|_{L^2} \leq c_4e^{-\frac{\kappa_1}{2} t}.
\end{equation}
Recalling the standard interpolation inequality (see \cite[(3.62)]{JinLW-CVPDE-2025}), for some positive constant $C_I$, it holds that
\begin{equation}\label{interpolation}
\|f\|_{L^\infty(\Omega)} \leq C_I \|f\|_{C^\alpha(\bar{\Omega})}^{\frac{n}{n+\alpha}} \|f\|_{L^1(\Omega)}^{\frac{\alpha}{n+\alpha}} \quad \text{for } f \in L^1(\Omega) \cap C^\alpha(\bar{\Omega}).
\end{equation}
By the fact that $W^{1,\infty}(\Omega) \hookrightarrow C^\alpha(\bar{\Omega})$ and \eqref{GB-1}, we obtain
\begin{equation}\label{holder-uw}
    \sup_{t\geq 1}(\|u(\cdot,t)\|_{C^{\alpha}}+\|w(\cdot,t)\|_{C^{\alpha}})\leq c_5
\end{equation}
Therefore, applying \eqref{interpolation} to $u-\bar{u}_0$ and noting \eqref{GST=L2-uv} with \eqref{holder-uw}, we obtain
\begin{equation}\label{GST=Linf-u}
\begin{aligned}
\|u(\cdot,t)-\bar{u}_0\|_{L^\infty}
&\le c_6
\|u(\cdot,t)-\bar{u}_0\|_{C^\alpha}^{\frac{n}{n+\alpha}}
\|u(\cdot,t)-\bar{u}_0\|_{L^1}^{\frac{\alpha}{n+\alpha}}\\
&\le c_6 |\Omega|^{\frac{\alpha}{2(n+\alpha)}}
\|u(\cdot,t)-\bar{u}_0\|_{C^\alpha}^{\frac{n}{n+\alpha}}
\|u(\cdot,t)-\bar{u}_0\|_{L^2}^{\frac{\alpha}{n+\alpha}}\\
&\le c_7 e^{-\kappa_2 t},\quad \forall t\ge1,
\end{aligned}
\end{equation}
where $\kappa_2 := \frac{\alpha \kappa_1}{2(n+\alpha)}$. Similarly, applying \eqref{conclu_uv} and \eqref{interpolation}, we deduce from \eqref{GST=L2-uv} and \eqref{GST=L2-w} that
\begin{equation*}
\|v(\cdot,t)-\bar{v}_0\|_{L^\infty} + \|w(\cdot,t)-w_c\|_{L^\infty} \leq c_8 e^{-\kappa_2 t}, \quad \forall t \geq 1,
\end{equation*}
which together with \eqref{GST=Linf-u} yields \eqref{GST=0-1}. This completes the proof.
\end{proof}
\begin{proof}[{\bf Proof of Theorem \ref{global_tau0_app}}] Theorem \ref{global_tau0_app} follows from Lemma \ref{GB} and Lemma \ref{GST=0}.
\end{proof}

\section{Hopf bifurcation and steady state patterns: Proof of Theorem \ref{NPS}}\label{LB} 
This section will show the absence of Hopf bifurcation and non-constant steady states.
\subsection{Spectrum of the linearized operator} In this subsection,  we shall establish the linear stability of $(\bar{u}_0,\bar{v}_0,w_c)$ by analyzing the spectrum of the non-local linearized operator. Before proceeding, we denote
$$X:=\{\phi\in W^{2,p}(\Omega): \ \nabla\phi\cdot\nu=0\ \text{on}\ \partial\Omega\},\ Y_0:=\Big\{\phi\in L^p(\Omega):\int_\Omega \phi=0\Big\}, \  X_0:=X\cap Y_0,$$
where {\color{black}$p>\max\{n,2\}$}. Motivated by \eqref{le*0-1}, we define a linear operator 
$$\mathcal{S}:D(\mathcal{S})=X_0\times X_0 \subset Y_0\times Y_0\rightarrow Y_0\times Y_0$$
by
\begin{equation}\label{ML}
\mathcal S
\begin{pmatrix}
		\phi  \\
		\psi \\
	\end{pmatrix}
=\begin{pmatrix}
		\Delta+\chi_1\lambda\bar u_0w_c\Delta(-d\Delta+J)^{-1} & \chi_1\lambda\bar u_0w_c\Delta(-d\Delta+J)^{-1} \\
		-\chi_2\bar v_0\Delta & \Delta \\
	\end{pmatrix}\begin{pmatrix}
		\phi  \\
		\psi \\
	\end{pmatrix}.
\end{equation}
Indeed, one can check that $$\Delta\in\mathcal L(X_0,Y_0), \ \ \ (-d\Delta+J)^{-1}\in\mathcal L(Y_0,X_0),\ \ \  \Delta(-d\Delta+J)^{-1}\in\mathcal L(Y_0,Y_0).$$ 
Therefore, $\mathcal S$ is well defined.  Here, \(\mathcal L(Z,S)\) denotes 
$$\mathcal{L}(Z,S):=\{ A| A ~\text{is~a~bounded~linear~operator~from}~Z~\text{to}~S\}$$
with $Z,S$ being normed linear spaces. {\color{black}However, the nonlocal structure $\Delta(-d\Delta+J)^{-1}$ of $\mathcal S$ prevents a direct application of the classical spectral theory for the Neumann Laplacian. To develop an appropriate spectral framework,} we therefore establish the following properties of $\mathcal S$:
\begin{itemize}
\item $\mathcal S$ is closed and $D(\mathcal{S})$ is dense in $Y_0\times Y_0$ (see Lemma \ref{cp});
\item $\mathcal S$ generates an analytic $C_0$-semigroup (see Lemma \ref{analytic}), so that its spectral bound determines the exponential growth of the linearized flow;
\item $\mathcal S$ has compact resolvent, so that the spectrum of $\mathcal S$ consists only of eigenvalues with finite algebraic multiplicity (see Lemma \ref{spectrum}).
\end{itemize}
\vspace{1.5mm}
We first show that the linear operator $\mathcal{S}$ is closed and $D(\mathcal S)$ is dense in $Y_0\times Y_0$.
\begin{lemma}\label{cp}
Let the linear operator $\mathcal{S}$ be defined in \eqref{ML}. Then the linear operator $\mathcal{S}$ is closed.  Moreover, the domain $D(\mathcal S)$ is dense in $Y_0\times Y_0$.
\end{lemma}
\begin{proof}
Define $L_0: D(L_0)=X_0\times X_0\subset Y_0\times Y_0\rightarrow Y_0\times Y_0 $ and $P_0: D(P_0)=Y_0\times Y_0\rightarrow Y_0\times Y_0$, where
$$L_0:=\begin{pmatrix}
		\Delta & 0 \\
		-\chi_2\bar{v}_0\Delta & \Delta \\
	\end{pmatrix}, \ \ 
    P_0:=\begin{pmatrix}
    \chi_1\lambda\bar u_0w_c\Delta(-d\Delta+J)^{-1}&\chi_1\lambda\bar u_0w_c\Delta(-d\Delta+J)^{-1}\\
    0 &0\\
    \end{pmatrix}.
$$
Then $\mathcal{S}:=L_0+P_0\big|_{X_0\times X_0}.$ Let $h:=(h_1,h_2)^\mathcal{T}$,  and denote
$$\|h\|_{S\times S}=\|h_1\|_S+\|h_2\|_S,$$
where $\mathcal{T}$ denotes the transpose, $S$ denotes the Banach space.  We first prove that $\mathcal{S}$ is closed, which will be divided into three steps.

{\bf Step 1: The linear operator $L_0$ is closed}. By the definition of closed operator, it suffices to show that for any $f^i\in D(L_0)$ satisfying
\begin{equation*}
\lim\limits_{i\to\infty}(\| f^i-f\|_{Y_0\times Y_0}+\| L_0f^i-g\|_{Y_0\times Y_0})=0,
\end{equation*}
then $f\in D(L_0)~\text{and}~g=L_0 f$,
where $f^i:=(f_1^i,f_2^i)^\mathcal{T},\ \ f:=(f_1,f_2)^\mathcal{T},\ \  g:=(g_1,g_2)^\mathcal{T}.$

Using the definition of $L_0$, we can derive that $\| L_0f^i-g\|_{Y_0\times Y_0}\rightarrow 0~\text{as}~i\rightarrow\infty$ can be rewritten as
\begin{equation}\label{fact2}
\lim\limits_{i\rightarrow\infty}(\|\Delta f_1^i-g_1\|_{Y_0} + \|\Delta f_2^i-\chi_2\bar{v}_0\Delta f_1^i -g_2\|_{Y_0})=\lim\limits_{i\rightarrow\infty}\| L_0f^i-g\|_{Y_0\times Y_0}=0.
\end{equation}
Since $f_1^i,f_1^j\in X_0$, we can apply the standard $L^p$-estimate for elliptic equations to derive 
\begin{equation}\label{ellip1}
\| f_1^i-f_1^j\|_{W^{2,p}}\leq c_1 \| \Delta f_1^i-\Delta f_1^j\|_{L^p},
\end{equation}
which along with the fact \eqref{fact2} yields that $\{f_1^i\}$ is a Cauchy sequence in $W^{2,p}.$ Since the Sobolev space $W^{2,p}$ is complete, there exists a $\tilde{f}_1\in W^{2,p}$ such that 
\begin{equation}\label{fw2p}
\lim\limits_{i\rightarrow\infty}\| f_1^i-\tilde{f}_1\|_{W^{2,p}}=0,
\end{equation}
this combined with the Sobolev embedding $W^{2,p}\hookrightarrow L^p$ yields 
\begin{equation}\label{flp}
\lim\limits_{i\rightarrow\infty}\| f_1^i-\tilde{f}_1\|_{L^p}=0, 
\end{equation}
hence $\tilde{f}_1=f_1$. Similarly, using \eqref{ellip1}, one has
\begin{equation*}
\| f_2^i-f_2^j\|_{W^{2,p}}\leq c_2 \| (\Delta f_2^i-\chi_2\bar{v}_0\Delta f_1^i) -(\Delta f_2^j-\chi_2\bar{v}_0\Delta f_1^j)+\chi_2\bar{v}_0(\Delta f_1^i-\Delta f_1^j)\|_{L^p},
\end{equation*}
which, together with \eqref{fact2} and \eqref{ellip1}, implies that $\{f_2^i\}$ is a Cauchy sequence in $W^{2,p}.$ Then we can find a $\tilde{f}_2\in W^{2,p}$ such that 
\begin{equation}\label{fw2p2}
\lim\limits_{i\rightarrow\infty}\| f_2^i-\tilde{f}_2\|_{W^{2,p}}=0,
\end{equation}
thus $\tilde{f}_2=f_2$ and $f=(f_1,f_2)\in W^{2,p}\times W^{2,p}.$

Next, we show $\nabla f_1\cdot\nu\big |_{\partial\Omega}=0$ and $\int_\Omega f_1=0.$ Noting the fact that $\nabla f_1^i\cdot\nu \big |_{\partial\Omega}=0$ and applying the trace interpolation inequality \cite[Lemma 2.5]{JinCH-JDE-2020} gives
\begin{equation*}
    \big\|\nabla f_1\cdot\nu\big\|_{L^p(\partial\Omega)}^p = \left\|\nabla f_1\cdot\nu-\nabla f_1^i\cdot\nu\right\|_{L^p(\partial\Omega)}^p \leq \| \nabla (f_1^i - f_1) \|_{L^p(\partial\Omega)}^p \leq c_1 \| f_1^i - f_1 \|_{L^p(\Omega)}^{\frac{p-1}{2}} \| f_1^i - f_1\|_{W^{2,p}(\Omega)}^{\frac{p+1}{2}},
\end{equation*}
which combined with \eqref{fw2p}-\eqref{flp} and $\tilde{f}_1=f_1$ yields $\int_{\partial\Omega} |\nabla f_1\cdot\nu|^pdS=0$ and hence $\nabla f_1\cdot\nu\big |_{\partial\Omega}=0$. By the facts $\int_\Omega f_1^i=0$, \eqref{flp} and $\tilde{f}_1=f_1$, one has 
$$  \bigg|\int_\Omega f_1\bigg|=\bigg|\int_\Omega (f_1^i-f_1)\bigg|\leq \int_\Omega | f_1^i-f_1|\leq |\Omega|^{\frac{p-1}{p}}\|f_1-f_1^i\|_{L^p}\rightarrow 0~\text{as}~i\rightarrow\infty,$$
this gives $\int_\Omega f_1=0.$ Therefore, $f_1\in X_0.$ Similarly, one derives $f_2\in X_0.$ Then, it holds that $f\in D(L_0).$ Finally, by \eqref{fw2p} and $\tilde{f}_1=f_1$ again, we have
\begin{equation}\label{f1-1}
\lim\limits_{i\rightarrow\infty} \|\Delta f_1^i-\Delta f_1\|_{L^p}=0,
\end{equation}
this along with the fact  $\lim\limits_{i\rightarrow\infty}\| \Delta f_1^i-g_1\|_{L^p}=0$ in \eqref{fact2} gives $\Delta f_1=g_1$. 
Following from \eqref{fw2p2} and $\tilde{f}_2=f_2$, one obtains
\begin{equation}\label{f2-1}
\lim\limits_{i\rightarrow\infty}\|\Delta f_2^i-\Delta f_2\|_{L^p}=0.
\end{equation}
Combining \eqref{f1-1} with \eqref{f2-1} yields that 
$$\| (\Delta f_2^i-\chi_2\bar{v}_0\Delta f_1^i) -(\Delta f_2-\chi_2\bar{v}_0\Delta f_1)\|_{L^p}\leq \|\Delta f_2^i-\Delta f_2\|_{L^p}+\chi_2\bar{v}_0\|\Delta f_1^i-\Delta f_1\|_{L^p}\rightarrow 0$$
as $i\rightarrow\infty$. This together with \eqref{fact2} gives $g_2=\Delta f_2-\chi_2\bar{v}_0\Delta f_1.$ Therefore, $g=L_0 f$ and the linear operator $L_0$ is closed.

{\bf Step 2: The linear operator $P_0$ is bounded.} By the $L^p$-estimates for elliptic equations, for any $z\in Y_0$, one has
\begin{equation*}
\|(-d\Delta+J)^{-1}z\|_{W^{2,p}}\leq c_3\|z\|_{L^p},
\end{equation*}
which implies the linear operator $(-d\Delta+J)^{-1}: Y_0\rightarrow X_0\hookrightarrow Y_0$ is bounded. Similarly, we have 
\begin{equation*}
\|\chi_1\lambda \bar{u}_0w_c\Delta (-d\Delta+J)^{-1}z\|_{L^p}\leq c_4\| (-d\Delta+J)^{-1}z\|_{W^{2,p}}\leq c_5\|z\|_{L^p}.
\end{equation*}
This shows that the linear operator $P_0:Y_0\times Y_0\rightarrow Y_0\times Y_0$ is bounded, i.e., 
\begin{equation}\label{P0B*}
\|P_0 f\|_{Y_0\times Y_0}\leq c_6 \|f\|_{Y_0\times Y_0}.
\end{equation}

{\bf Step 3: The linear operator $\mathcal{S}$ is closed.} By the definition of closed operator, it suffices to show that for any $f^i\in D(\mathcal{S})=D(L_0)\subset D(P_0)$, 
\begin{equation*}
\| f^i-f\|_{Y_0\times Y_0}\rightarrow 0,\| \mathcal{S}f^i-g\|_{Y_0\times Y_0}\rightarrow 0~\text{as}~i\rightarrow\infty\Rightarrow f\in D(\mathcal{S})~\text{and}~g=\mathcal{S}f.
\end{equation*}
By $\| \mathcal{S}f^i-g\|_{Y_0\times Y_0}\rightarrow 0~\text{as}~i\rightarrow\infty$ and \eqref{P0B*}, one derives
\begin{equation*}
\begin{split}
\|L_0 f^i-(g-P_0f)\|_{Y_0\times Y_0}
&=\|(\mathcal{S}-P_0)f^i-(g-P_0f)\|_{Y_0\times Y_0}\\
&\leq \| \mathcal{S}f^i-g\|_{Y_0\times Y_0} +\|P_0f^i-P_0f\|_{Y_0\times Y_0}\\
&\leq \| \mathcal{S}f^i-g\|_{Y_0\times Y_0}+c_6\|f^i-f\|_{Y_0\times Y_0}\rightarrow 0,
\end{split}
\end{equation*}
which, along with $D(\mathcal{S})=D(L_0)$ and the fact that $L_0$ is closed, indicates
$$f\in D(L_0)=D(\mathcal{S})~\text{and}~ g-P_0f=L_0f.$$
This shows $g=\mathcal{S}f$, hence the linear operator $\mathcal{S}: D(\mathcal{S})=X_0\times X_0\rightarrow Y_0\times Y_0$ is closed.

{\color{black}Finally, it remains to show that $D(\mathcal{S})$ is dense in $Y_0\times Y_0$. Since $C_c^\infty(\Omega)\cap Y_0\subset X_0$ and $C_c^\infty(\Omega)\cap Y_0$ is dense in $Y_0$, it follows that $X_0$ is dense in $Y_0$. Hence, $D(\mathcal{S})=X_0\times X_0$ is dense in $Y_0\times Y_0$.}
\end{proof}

\begin{lemma}\label{analytic}
The operator $\mathcal S$ generates an analytic $C_0$-semigroup on $Y_0\times Y_0$.
\end{lemma}
\begin{proof}
It is well known that the Neumann Laplacian 
$$\Delta: D(\Delta)=X_0\subset Y_0\rightarrow Y_0$$
is closed and $D(\Delta)$ is dense in $Y_0$. Moreover, $\Delta$ generates a bounded analytic $C_0$-semigroup on $Y_0$. In particular, $\Delta$ is sectorial, i.e., there exist $\theta\in(\pi/2,\pi)$ and $c_0>0$ independent of $\zeta$ such that
\begin{equation}\label{0714}
\Sigma_\theta\subset\rho(\Delta),
\qquad
\|(\zeta I-\Delta)^{-1}\|_{\mathcal L(Y_0)}
\leq \frac {c_0}{|\zeta|}
\quad\text{for all }\zeta\in\Sigma_\theta,
\end{equation}
where $\Sigma_\theta
:=
\{\zeta\in\mathbb C:|\arg\zeta|<\theta\}\setminus\{0\}$, and the set
\begin{equation*}
 \rho(A):=\{ \zeta\in \mathbb{C}\ | \  (\zeta I-A)^{-1}\in \mathcal{L}(Z,Z)\}
\end{equation*}
denotes the resolvent set of $A$, with $Z$ being a Banach space and $A: D(A)\subset Z\rightarrow Z$ being a closed linear operator.
By the proof of Lemma \ref{cp}, $L_0$ and $\mathcal{S}$ are closed and densely defined on $Y_0\times Y_0$. We denote their resolvent sets by $\rho(L_0)$ and $\rho(\mathcal S)$, respectively.  The following proof is divided into three steps.

{\bf Step 1.}
Fix $\zeta\in\Sigma_\theta$ and let $f,g\in Y_0$ be given. We consider the following system
\begin{equation}\label{0708}
\begin{cases}
 (\zeta I-\Delta)U=f, &x\in\Omega,\\
 \chi_2\bar{v}_0\Delta U+(\zeta I-\Delta)V=g, &x\in\Omega.
\end{cases}
\end{equation}
Since $\zeta\in \Sigma_\theta\subset\rho(\Delta),$  the first equation in \eqref{0708} admits the unique solution $U=(\zeta I-\Delta)^{-1}f \in X_0.$ Then $\Delta U\in Y_0$, and the second equation in \eqref{0708} admits the unique solution $$V=(\zeta I-\Delta)^{-1}\bigl(g-\chi_2\bar v_0\Delta U\bigr)\in X_0.$$
Therefore,  the linear operator $\zeta I-L_0:X_0\times X_0\rightarrow Y_0\times Y_0$ is invertible, and we denote its inverse by $(\zeta I-L_0)^{-1}:Y_0\times Y_0\rightarrow  X_0\times X_0$. Furthermore, the standard $L^p$-elliptic estimates yield
$$\|U\|_{W^{2,p}}\leq c_1(\zeta) \|f\|_{L^p},$$
and $$\|V\|_{W^{2,p}}\le c_2(\zeta)\|g-\chi_2\bar v_0\Delta U\|_{L^p}\le c_3(\zeta)\bigl(\|f\|_{L^p}+\|g\|_{L^p}\bigr),$$
which yields that the linear operator $(\zeta I-L_0)^{-1}: Y_0\times Y_0\rightarrow  X_0\times X_0$ is bounded. {\color{black}Here, positive constants $c_i(\zeta)(i=1,2,3)$ are dependent on $\zeta$.} Moreover, this inverse operator is given by
\begin{equation*}
(\zeta I-L_0)^{-1}=\begin{pmatrix}
		(\zeta I-\Delta)^{-1} & 0 \\
		-\chi_2\bar{v}_0(\zeta I-\Delta)^{-1}\Delta (\zeta I-\Delta)^{-1} & (\zeta I-\Delta)^{-1} \\
	\end{pmatrix}.
\end{equation*}

{\bf Step 2.} By Step 1 and the fact that $X_0\times X_0\hookrightarrow Y_0\times  Y_0$ continuously, we derive 
$$(\zeta I-L_0)^{-1}\in \mathcal{L}(Y_0\times Y_0).$$
Using the identity
\begin{equation*}
 \Delta(\zeta I-\Delta)^{-1}=\zeta (\zeta I-\Delta)^{-1}-I,
\end{equation*}
we deduce from  \eqref{0714} that
\begin{equation}\label{0713-1}
\begin{split}
&\| (\zeta I-\Delta)^{-1}g-\chi_2\bar{v}_0(\zeta I-\Delta)^{-1}\Delta (\zeta I-\Delta)^{-1}f\|_{L^p}\\
&=\|(\zeta I-\Delta)^{-1}[g-\chi_2\bar{v}_0\zeta (\zeta I-\Delta)^{-1}f+\chi_2\bar{v}_0f]\|_{L^p}\\
&\leq \frac{c_0}{|\zeta|}(\|g\|_{L^p}+\chi_2\bar{v}_0|\zeta|\frac{c_0}{|\zeta|}\|f\|_{L^p}+\chi_2\bar{v}_0\|f\|_{L^p})\\
&\leq \frac{c_4}{|\zeta|}(\|f\|_{L^p}+\|g\|_{L^p}),
\end{split}
\end{equation}
where $c_4:=c_0(\chi_2\bar{v}_0c_0+\chi_2\bar{v}_0+1)$ is independent of $\zeta$. Combining \eqref{0714} with \eqref{0713-1}, we obtain 
\begin{equation*}
\|(\zeta I-L_0)^{-1}\|_{\mathcal{L}(Y_0\times Y_0)}\leq \frac{c_5}{|\zeta|},
\end{equation*}
where the constant $c_5>0$ is independent of $\zeta$. Hence, the closed linear operator $L_0:D(L_0)=X_0\times X_0\subset Y_0\times Y_0\rightarrow Y_0\times Y_0$ is sectorial. By \cite[Theorem 4.6]{Engel-Nagel-2000}, $L_0$ generates an analytic $C_0$-semigroup on $Y_0\times Y_0$.

{\bf Step 3.} By \eqref{P0B*}, there exists $c_6>0$ such that
$$\|P_0 f\|_{Y_0\times Y_0}\leq c_6 \|f\|_{Y_0\times Y_0}, \quad\forall f\in D(L_0)\subset Y_0\times Y_0.$$
Thus, $P_0: Y_0\times Y_0\rightarrow Y_0\times Y_0$ is $L_0$-bounded with $L_0$-bound zero  (see \cite[Definition 2.1, p.169]{Engel-Nagel-2000}).

Therefore, the analytic-semigroup perturbation theorem (\cite[Theorem 2.10, p.176]{Engel-Nagel-2000}) implies that $\mathcal S$ generates an analytic $C_0$-semigroup on $Y_0\times Y_0$.
\end{proof}

Based on Lemmas \ref{cp}-\ref{analytic}, it follows from \cite[Corollary 3.12, p.281, and (1.7), p.299]{Engel-Nagel-2000} that the linear stability of $(\bar{u}_0,\bar{v}_0,w_c)$ for \eqref{system} is determined by the spectral bound 
\begin{equation*}
s(\mathcal S):=\sup \{\operatorname{Re}\eta:\eta\in \sigma (\mathcal S)\},
\end{equation*}
where $\sigma (\mathcal S)$ denotes the spectrum of the linear operator $\mathcal{S}$.  More precisely, $(\bar{u}_0,\bar{v}_0,w_c)$ is linearly stable if $s(\mathcal S)<0$, and linearly unstable if $s(\mathcal S)>0$.

\vspace{1.5mm}
Now, we show that $\sigma(\mathcal{S})$ consists only of eigenvalues with finite algebraic multiplicity. 
\begin{lemma}\label{spectrum}
Let the linear operator  $\mathcal{S}$ be defined in \eqref{ML}. Then the spectrum of $\mathcal{S}$ consists only of eigenvalues with finite algebraic multiplicity.
\end{lemma}
\begin{proof}
Choose $\zeta_0\in\Sigma_\theta$ with $|\zeta_0|$ sufficiently large. By Step 2 in the proof of Lemma \ref{analytic}, the linear operator $(\zeta_0 I-L_0)^{-1}$ satisfies
\begin{equation}\label{0713-2*}
\|(\zeta_0 I-L_0)^{-1}\|_{\mathcal{L}(Y_0\times Y_0)}\leq \frac{c_1}{|\zeta_0|},
\end{equation}
where the constant $c_1>0$ is independent of $\zeta_0$. Moreover, it  follows from \eqref{P0B*} that  the linear operator $P_0: Y_0\times Y_0\rightarrow Y_0\times Y_0$ is bounded. Hence, the linear operator $P_0(\zeta_0 I-L_0)^{-1}:  Y_0\times Y_0\rightarrow Y_0\times Y_0$ is bounded, i.e., $P_0(\zeta_0 I-L_0)^{-1}\in \mathcal{L}(Y_0\times Y_0)$. Then we deduce from \eqref{0713-2*} that
\begin{equation*}
\| P_0(\zeta_0 I-L_0)^{-1}f\|_{Y_0\times Y_0}\leq c_2 \|(\zeta_0 I-L_0)^{-1}f\|_{Y_0\times Y_0}\leq  \frac{c_3}{|\zeta_0|}\|f\|_{Y_0\times Y_0},
\end{equation*}
which implies 
$$\| P_0(\zeta_0 I-L_0)^{-1}\|_{\mathcal{L}(Y_0\times Y_0)}\leq \frac{c_3}{|\zeta_0|}<1.$$
Therefore, the Neumann series theorem yields that the bounded operator $I-P_0(\zeta_0 I-L_0)^{-1}:Y_0\times Y_0\rightarrow Y_0\times Y_0$ is invertible, and the inverse operator $[I-P_0(\zeta_0 I-L_0)^{-1}]^{-1}$ is bounded.

A direct computation shows that, on $D(\mathcal{S})=X_0\times X_0$,  
\begin{equation*}
\zeta_0 I-\mathcal{S}=\big[I-P_0(\zeta_0 I-L_0)^{-1}\big](\zeta_0 I-L_0).
\end{equation*}
Since $(\zeta_0 I-L_0)^{-1}: Y_0\times Y_0\rightarrow X_0\times X_0$ and $[I-P_0(\zeta_0 I-L_0)^{-1}]^{-1}:Y_0\times Y_0\rightarrow Y_0\times Y_0$ are bounded, it follows that $\zeta_0 I-\mathcal{S}$ is invertible. Moreover, the inverse is given by
\begin{equation*}
(\zeta_0 I-\mathcal{S})^{-1}=(\zeta_0 I-L_0)^{-1}\big[I-P_0(\zeta_0 I-L_0)^{-1}\big]^{-1}.
\end{equation*}
It maps $Y_0\times Y_0$ into $X_0\times X_0$ and satisfies
$$\| (\zeta_0I -\mathcal{S})^{-1}{f}\|_{X_0\times X_0}\leq c_4 \|{ f}\|_{Y_0\times Y_0}, \quad \forall{ f}\in Y_0\times Y_0.$$
Since $X_0\times X_0\hookrightarrow Y_0\times Y_0$ continuously, we obtain $(\zeta_0 I-\mathcal{S})^{-1}\in \mathcal{L}(Y_0\times Y_0)$. Hence, $\zeta_0\in \rho (\mathcal S)$, and $\rho (\mathcal S)\not=\emptyset$. Moreover, by the compact embedding $X_0\times X_0\hookrightarrow\hookrightarrow Y_0\times Y_0$, we know that
$$
(\zeta_0I-\mathcal S)^{-1}:Y_0\times Y_0\to Y_0\times Y_0
$$
is compact.  Thus,  $\mathcal S: D(\mathcal{S})=X_0\times X_0\subset Y_0\times Y_0\rightarrow Y_0\times Y_0$ has compact resolvent on $Y_0\times Y_0$. This alongside the compact resolvent theory directly finishes the proof of Lemma \ref{spectrum}.
\end{proof}

According to Lemma \ref{spectrum}, the linear stability of $(\bar{u}_0,\bar{v}_0,w_c)$ can be characterized through the eigenvalue problem associated with $\mathcal S$. Hence, it suffices to consider 
\begin{equation}\label{ep0}
\begin{cases}
\Delta \phi+\chi_1\bar{u}_0\lambda w_c\Delta [(-d\Delta+J)^{-1}(\phi+\psi)]=\eta \phi,&x\in\Omega,\\
\Delta\psi-\chi_2\bar{v}_0\Delta\phi=\eta \psi,&x\in\Omega,\\
\nabla\phi\cdot\nu=\nabla\psi\cdot\nu=0, &x\in\partial\Omega,\\
\int_\Omega \phi(x)=\int_\Omega\psi(x)=0.
\end{cases}
\end{equation}
\begin{lemma}\label{linear_PE}
Let $\{\sigma_m\}_{m\geq0}$ be the sequence of eigenvalues of $-\Delta$ under homogeneous Neumann boundary conditions, satisfying $
0=\sigma_0<\sigma_1\leq\sigma_2\leq\cdots,$ and let $y_m(x)$ be a corresponding complete orthonormal basis of $L^2(\Omega)$ consisting of Neumann eigenfunctions.  Define
\begin{equation}\label{Am}
\mathcal{G}_m=
\begin{pmatrix}
-\sigma_m-\dfrac{\chi_1\lambda \bar u_0 w_c\,\sigma_m}
{d\sigma_m+J}
&
-\dfrac{\chi_1\lambda \bar u_0 w_c\,\sigma_m}
{d\sigma_m+J}
\\[3mm]
\chi_2\bar v_0\sigma_m
&
-\sigma_m
\end{pmatrix}.
\end{equation}
{\color{black}Then $\eta$ is an eigenvalue of \eqref{ep0} if and only if $\eta$ is an eigenvalue of $\mathcal{G}_m$ for some $m\in \mathbb{N}^+$.}

\end{lemma}
\begin{proof}
By the assumptions in Lemma \ref{linear_PE}, we have
$$
-\Delta y_m=\sigma_m y_m \quad \text{in}\ \Omega;\quad
\nabla y_m\cdot\nu=0\quad\text{on }\partial\Omega.
$$
Since $\{y_m\}_{m\ge0}$ forms a complete orthonormal basis of $L^2(\Omega)$ and $Y_0\subset L^2(\Omega)$, for any $\phi,\psi\in Y_0$, we expand
\begin{equation}\label{0713pm1}
\phi=\sum_{m=0}^\infty a_m y_m,
\quad
\psi=\sum_{m=0}^\infty b_m y_m.
\end{equation}
Then one obtains
\begin{equation}\label{0713pm2}
(-d\Delta+J)^{-1}(\phi+\psi)= \sum_{m=0}^\infty
\frac{(a_m+b_m)}{d\sigma_m+J}
y_m.
\end{equation}

Now, let $\eta$ be an eigenvalue of \eqref{ep0} with an associated
eigenvector
$$
(\phi,\psi)\in D(\mathcal S)=X_0\times X_0,
\qquad
(\phi,\psi)\neq(0,0),
$$
then there exists some $m\in \mathbb{N}^+$ such that $(a_m,b_m)\not=(0,0)$. Substituting \eqref{0713pm1} and \eqref{0713pm2} into \eqref{ep0}, taking the $L^2$-inner product of both equations with $y_m$, and using the orthonormality of $\{y_m\}_{m\ge0}$, we obtain
\begin{equation}\label{0713pm3}
\begin{cases}
\eta a_m
&=-\sigma_m a_m-\frac{\chi_1\lambda\bar u_0w_c\sigma_m}
{d\sigma_m+J}
(a_m+b_m),\\
\eta b_m
&=-\sigma_m b_m+\chi_2\bar v_0\sigma_m a_m.
\end{cases}
\end{equation}
Then \eqref{0713pm3} is equivalent to
$$
\mathcal{G}_m
\binom{a_m}{b_m}=\eta \binom{a_m}{b_m},
$$
where $\mathcal{G}_m$ is defined in \eqref{Am}. Hence $\eta$ is an eigenvalue of $\mathcal{G}_m$.

Now we assume that $\eta$ is an eigenvalue of $\mathcal{G}_m$ for some $m\in\mathbb{N}^+$ with associated eigenvector $(a_m,b_m)\not=(0,0)$. Let $\phi=a_m y_m$ and $\psi=b_m y_m$, then $\int_\Omega \phi=\int_\Omega\psi=0$ and $\nabla\phi\cdot\nu|_{\partial\Omega}=\nabla\psi\cdot\nu|_{\partial\Omega}=0$. Moreover,  we derive
\begin{equation*}
\begin{split}
\Delta \phi+\chi_1\bar{u}_0\lambda w_c\Delta [(-d\Delta+J)^{-1}(\phi+\psi)]=-a_m\sigma_m y_m- \frac{\chi_1\lambda\bar{u}_0w_c\sigma_m}{d\sigma_m+J}(a_m+b_m)y_m=\eta \phi,
\end{split}
\end{equation*}
and 
\begin{equation*}
\Delta\psi-\chi_2\bar{v}_0\Delta\phi=-\sigma_m b_my_m+\chi_2\bar v_0\sigma_m a_my_m=\eta \psi,
\end{equation*}
which implies that $\eta$ is an eigenvalue of \eqref{ep0}. This completes the proof of Lemma \ref{linear_PE}.  
\end{proof}

Based on Lemma \ref{linear_PE}, we have the following linear stability results for \eqref{system}.
\begin{lemma}\label{SPE}
Let $\chi_1,\chi_2,\lambda,\mu,d,\bar{v}_0,\bar{u}_0$ be fixed. Then for all $r>0$, the constant steady state $(\bar u_0,\bar v_0,w_c)$ is linearly stable.
\end{lemma}
\begin{proof}
We deduce from Lemma \ref{linear_PE} that the linear stability of $(\bar u_0,\bar v_0,w_c)$ for \eqref{system} is determined by the eigenvalues of $\mathcal{G}_m$. Calculating directly yields the characteristic equation:
{\color{black}
$$ \eta^2+\sigma_m\big(2+g_1\big)\eta+\sigma_m^2\big(1+g_1+g_2\big)=0,$$
where 
\begin{align*}
g_1:=g_1(r,\sigma_m):=\frac{\chi_1\lambda \bar{u}_0w_c}{\sigma_m d+J}>0,\quad g_2:=g_2(r,\sigma_m):=\frac{\chi_1\chi_2\lambda \bar{u}_0\bar{v}_0w_c}{\sigma_m d+J}>0.\\
\end{align*}
Then some calculations yield
$$\sigma_m^2\big(2+g_1\big)^2-4\sigma_m^2\big(1+g_1+g_2\big)=\sigma_m^2(g_1^2-4g_2).$$
If $\sigma_m^2(g_1^2-4g_2)\leq 0$ for any fixed $m\in \mathbb{N}^+$, then
\begin{equation}\label{0722}
\operatorname{Re}\eta=-\frac{\sigma_m(2+g_1)}{2}\leq -\sigma_m\leq-\sigma_1<0.
\end{equation}
If $\sigma_m^2(g_1^2-4g_2)>0$ for any fixed $m\in \mathbb{N}^+$, then 
\begin{equation}\label{0722-1}
\eta^-:=\frac{\sigma_m}{2}\bigg(-(2+g_1)-\sqrt{g_1^2-4g_2}\bigg)\leq -\sigma_m\leq-\sigma_1<0,
\end{equation}
and
\begin{equation}\label{0722-2}
\begin{split}
\eta^+:=\frac{\sigma_m}{2}\bigg(-(2+g_1)+\sqrt{g_1^2-4g_2}\bigg)\leq \frac{\sigma_m}{2}\big(-(2+g_1)+g_1\big)\leq -\sigma_m\leq-\sigma_1<0.
\end{split}
\end{equation}
Therefore, by \eqref{0722}, \eqref{0722-1}, and
\eqref{0722-2}, for every $m\in\mathbb N^+$ and every
$\eta\in\sigma(\mathcal G_m)$, we have
$$
\operatorname{Re}\eta\leq-\sigma_1<0.
$$
From Lemma \ref{spectrum} and Lemma \ref{linear_PE}, one obtains
$$
s(\mathcal S)
=
\sup_{m\in\mathbb N^+}
\max_{\eta\in\sigma(\mathcal G_m)}
\operatorname{Re}\eta
\leq-\sigma_1<0.
$$
Hence, $(\bar u_0,\bar v_0,w_c)$ is linearly stable
with respect to \eqref{system} for every $r>0$.
}
\end{proof}

Next, we shall show that  the system \eqref{system} does not admit positive non-constant steady state solutions. 

\begin{lemma}\label{NS}
Let $\chi_1,\chi_2,\lambda,\mu,d,\bar{v}_0,\bar{u}_0$ be fixed. Then for all $r>0$,  the system \eqref{system} does not admit any positive non-constant steady state solution.
\end{lemma}
\begin{proof}
Assume that $(U,V,W)$ is a positive steady state  solution to \eqref{system}. Then it satisfies the following system 
\begin{equation}\label{s-system}
\begin{cases}
0=\Delta U-\chi_1\nabla\cdot(U\nabla W),&x\in\Omega,\\
0=\Delta V-\chi_2\nabla\cdot(V\nabla U),&x\in\Omega,\\
0=d\Delta W-\lambda(U+V)W-\mu W+r,&x\in\Omega,\\
\nabla U\cdot\nu=\nabla V\cdot\nu=\nabla W\cdot \nu=0,&x\in\partial\Omega.\\
\end{cases}
\end{equation}
Let $Z_1:=Ue^{-\chi_1W}>0$, then the $U$-equation can be rewritten as
\begin{equation}\label{ss-1}
\nabla \cdot (e^{\chi_1 W}\nabla Z_1)=0,\ \ x\in\Omega;\quad \nabla Z_1\cdot\nu=0,\ \ x\in\partial\Omega.
\end{equation}
Multiplying \eqref{ss-1} by $Z_1$ and integrating the result by parts, we derive
$$ \int_\Omega e^{\chi_1W}|\nabla Z_1|^2=0,$$
which implies $\nabla Z_1\equiv 0$, hence $Z_1\equiv c_1$ for some  constant $c_1>0$. Then we have
$$U = c_1 e^{\chi_1 W}.$$
Similarly, we find some positive constant $c_2$ such that
\begin{equation*}
     V = c_2 e^{\chi_2 U} = c_2 e^{\chi_2 c_1 e^{\chi_1 W}}.
\end{equation*}
Then the third equation in \eqref{s-system} yields 
\begin{equation}\label{ss-w}
    d\Delta W - f(W)=0, 
\end{equation}
where
$$f(W) := \lambda(c_1 e^{\chi_1 W} + c_2 e^{\chi_2 c_1 e^{\chi_1 W}})W + \mu W - r.$$
Noting that $\int_\Omega (W - \bar{W}) \, = 0$ and $f(\bar{W})$ is a constant, it naturally holds that
\begin{equation}\label{fact0}
    \int_\Omega f(\bar{W})(W - \bar{W})  = 0.
\end{equation}
We multiply \eqref{ss-w} by $W - \bar{W}$, integrate the result by parts over $\Omega$, and use \eqref{fact0} to get
\begin{equation}\label{ss-w2}
\begin{split}
    -d\int_\Omega |\nabla W|^2 &= \int_\Omega f(W)(W - \bar{W}) 
    = \int_\Omega (f(W) - f(\bar{W}))(W - \bar{W}).
\end{split}
\end{equation}
Since $W>0$, a direct computation shows that $f'(W) > 0$. Thus, we have 
$$(f(W) - f(\bar{W}))(W - \bar{W}) \geq 0,$$
then it follows from \eqref{ss-w2} that $\int_\Omega |\nabla W|^2=0$. This implies \(\nabla W\equiv0\), and hence \(W\) is constant in \(\Omega\). Consequently, \(U=c_1e^{\chi_1W}\) and \(V=c_2e^{\chi_2U}\) are also constants. Therefore, no positive non-constant steady state exists.
\end{proof}

\begin{proof}[{\bf Proof of Theorem \ref{NPS}}] Theorem \ref {NPS} follows directly from Lemma \ref {SPE} and Lemma \ref {NS}.
\end{proof}

\bigbreak
\noindent \textbf{Acknowledgment}.
The research of H.Y. Jin was supported by the NSF of China (No. 12371203), Guangdong Major Project of Basic Research (2026B0303000003).

\small{

}
\end{document}